\documentclass[11 pt]{amsart}
\usepackage{bbm, mathrsfs, enumerate, amssymb, fourier}
\usepackage[foot]{amsaddr}

\usepackage[latin1]{inputenc}
\usepackage{amsfonts,amssymb,amsmath,amsthm}
\usepackage[headinclude,DIV13]{typearea}
\usepackage{hyperref}
\usepackage{mathrsfs,bbm,stmaryrd}

\usepackage{geometry}
\usepackage{graphicx, psfrag}

 \usepackage{amsmath,amsfonts,amssymb}
 \usepackage{bbold}
 \usepackage{color}
 \usepackage{setspace}

\newtheorem{rem}{Remark}
\newtheorem{lem}{Lemma}[section]
\newtheorem{pro}{Proposition}

\newtheorem{theo}{Theorem}

\newtheorem{ass}{Assumption}

\renewcommand{\P}{\mathbb{P}}
\newcommand{\R}{\mathbb{R}}

\newcommand{\E}{\mathbb{E}}

\newcommand{\N}{\mathbb{N}}
\newcommand{\Z}{\mathbb{Z}}

\newcommand{\eps}{\varepsilon}

\DeclareMathOperator{\var}{Var}
\DeclareMathOperator{\cov}{Cov}
\numberwithin{equation}{section}

\begin{document}

\title[Recombination and adaptation]{An Eco-Evolutionary approach of Adaptation and Recombination in a large population of varying size}

\author{Charline Smadi}
\address{Universit\'e Paris-Est, CERMICS (ENPC), F-77455 Marne La Vall\'ee, France and CMAP UMR 7641, \'Ecole Polytechnique CNRS, Route de Saclay,
91128 Palaiseau Cedex, France}
\email{charline.smadi@polytechnique.edu}

\begin{abstract}
We identify the genetic signature of a selective sweep in a population described by a birth-and-death process with density
dependent competition. We study the limit behaviour for large $K$, where $K$ scales the population size.
We focus on two loci: one under selection and one neutral. We distinguish a soft sweep occurring after an environmental change, from a hard sweep occurring 
after a mutation, and express the neutral proportion variation as a function of the
ecological parameters, recombination probability $r_K$, and $K$. We show that for a hard sweep, two recombination regimes appear according to the order 
of $r_K\log K$.
\end{abstract}

\maketitle

\section{Introduction}

There are at least two different ways of adaptation for a population: selection can either act on a new mutation 
(hard selective sweep), or on preexisting 
alleles that become advantageous after an environmental change (soft selective sweep from standing variation). 
New mutations are sources 
of diversity, and hard selective sweeps were until recently the only considered 
way of adaptation. Soft selective sweeps from standing variation allow a faster adaptation to novel environments, 
and their importance 
is growing in empirical and theoretical studies 
(Orr and Betancourt \cite{orr2001haldane}, Hermisson and Pennings \cite{hermisson2005soft}, Prezeworski, Coop and Wall 
\cite{prezeworski2005signature}, Barrett and Schluter \cite{barrett2008adaptation}, Durand and al  \cite{durand2010standing}).
In particular Messer and Petrov \cite{messer2013population} review a lot of evidence, from individual case
studies as well as from genome-wide scans, that soft sweeps (from standing variation and from recurrent mutations)
 are common in a broad range of organisms.
These distinct selective sweeps entail different 
genetic signatures in the vicinity of the novely 
fixed allele, and the multiplication of genetic data available allows one to detect these signatures in current populations as described by Peter, Huerta-Sanchez 
and Nielsen \cite{peter2012distinguishing}. To do this in an effective way, it is necessary to identify accurately the signatures left 
by these two modes of adaptation.
We will not consider in this work the soft selective sweeps from recurrent mutations. For a study of these sweeps we refer to 
\cite{pennings2006soft,pennings22006soft,hermisson2008pattern}.
\\

In this work, we consider a sexual haploid population of varying size, modeled by a birth and death process with density 
dependent competition. 
Each individual's ability to survive and reproduce depends on its own genotype and on the population state. 
More precisely, each individual is 
characterized by some ecological parameters: 
birth rate, intrinsic death rate and competition kernel describing the competition with other individuals depending on their genotype. 
The differential 
reproductive success of individuals generated by their interactions entails progressive variations in the number of individuals carrying a given genotype. 
This process, 
called natural selection, is a key mechanism of evolution.
Such an eco-evolutionary approach has been introduced by Metz and coauthors in \cite{metz1996adaptive} and made 
rigorous in the seminal paper of Fournier and M\'el\'eard \cite{fournier2004microscopic}.
Then it has been developed by 
Champagnat, M\'el\'eard and coauthors (see
\cite{champagnat2006microscopic,
champagnat2011polymorphic,champagnat2014adaptation} and references therein) for 
the haploid asexual case and by Collet, M\'el\'eard and Metz \cite{collet2011rigorous} 
and Coron and coauthors \cite{coron2013slow,coron2014stochastic} for the diploid sexual case.
The recent work of Billiard and coauthors \cite{billiard2013stochastic} studies the dynamics of a two-locus model in an haploid asexual population.
Following these works, we introduce a parameter $K$ called carrying capacity which scales the population size, and study the limit 
behavior for large $K$. But unlike them, we focus on two loci in a sexual haploid population and take into account recombinations: one locus is under selection and has 
two possible alleles $A$ and $a$ 
and the second one is neutral with allele $b_1$ or $b_2$. When two individuals give birth, either a recombination occurs with probability $r_K$ and 
the newborn inherits one allele from each parent, or he is the clone of one parent. \\
 
We first focus on a soft selective sweep from standing variation occurring after a change in the environment (new pathogen, environmental catastrophe, occupation of a new ecological 
niche,...). We assume that before the change the alleles $A$ and $a$ were neutral and both represented a positive fraction of the 
population, and that in the new environment the allele $a$ becomes favorable and goes to fixation.
We can divide the selective sweep in two periods: a first one where the population process is well approximated by the solution of a deterministic dynamical system, and a second one where $A$-individuals are 
near extinction, the deterministic 
approximation fails and the fluctuations of the $A$-population size become predominant.
We give the asymptotic value of the final neutral allele proportion as a 
function of the ecological parameters, recombination probability $r_K$ and solutions of a two-dimensional competitive Lotka-Volterra 
system. \\

We then focus on hard selective sweeps.
We assume that a mutant $a$ appears in a monomorphic $A$-population at ecological equilibrium. 
As stated by Champagnat in \cite{champagnat2006microscopic}, the selective sweep is divided in three periods: during the first one, the resident population size stays near its equilibrium value, 
and the mutant population size grows until it reaches a non-negligible fraction of the total population size. 
The two other periods are the ones 
described for the soft selective sweep from standing variation. Moreover, the time needed for the mutant $a$ to fix in the population is of order 
$\log K$. We prove that the distribution of neutral alleles at the end of the sweep has 
different shapes according to the order of the recombination probability per reproductive event $r_K$ with respect to $1/\log K$. More precisely, we 
find two recombination regimes: a strong one where $r_K \log K$  is large, and a weak one where $r_K\log K$ is bounded.
In both recombination regimes, we give the asymptotic value of the final neutral allele proportion as a 
function of the ecological parameters and recombination probability $r_K$.
In the strong recombination regime, the frequent exchanges of neutral alleles between the $A$ and $a$-populations yield an homogeneous neutral 
repartition in the two populations and the latter is not modified by the sweep.
 In the weak recombination regime, the frequency of the neutral allele carried by the first mutant increases because it is linked to the 
positively selected allele.
This phenomenon has been called genetic hitch-hiking 
by Maynard Smith and Haigh \cite{smith1974hitch}. \\

The first studies of hitch-hiking, initiated by Maynard Smith and Haigh \cite{smith1974hitch}, have modeled the mutant population size as 
the solution of a deterministic logistic equation \cite{ohta1975effect,kaplan1989hitchhiking,stephan1992effect,
stephan2006hitchhiking}. 
Kaplan and coauthors \cite{kaplan1989hitchhiking} described the neutral genealogies by a structured coalescent 
where the background was the frequency of the 
beneficial allele. 
Barton \cite{barton1998effect}
 was the first to point out the importance of the stochasticity of the mutant population size and the errors made by ignoring it.
He divided the sweep in four periods: the two last ones are the analogues of the two last steps described in \cite{champagnat2006microscopic}, and 
the two first ones 
correspond to the first one in \cite{champagnat2006microscopic}. 
Following the approaches of \cite{kaplan1989hitchhiking} and \cite{barton1998effect}, 
a series of works studied the genealogies of neutral alleles sampled at the end of the sweep and took into account the randomness 
of the mutant population size
during the sweep. In particular,
 Durrett and Schweinsberg \cite{durrett2004approximating,schweinsberg2005random},
 Etheridge and coauthors \cite{etheridge2006approximate}, Pfaffelhuber and Studeny \cite{pfaffelhuber2007approximating}, and Leocard \cite{stephanie2009selective}  
described the population process by a structured coalescent and finely studied genealogies of neutral alleles during the sweep.
Eriksson
and coauthors \cite{eriksson2008accurate} described a deterministic approximation for the growth of the beneficial allele
frequency during a sweep, which leads to more accurate approximation than previous models for
large values of the recombination probability.
Unlike our model, in all these works, the population size was constant and the individuals' ``selective value'' did not depend on the population state,
 but only on the individuals' genotype.\\

The structure of the paper is the following. In Section \ref{model} we describe the model, review some results of  
\cite{champagnat2006microscopic} about the two-dimensional population process when we do not consider the neutral locus, and present the main results.
 In Section \ref{prel_results} we state a semi-martingale decomposition of neutral proportions, 
a key tool in the different proofs. Section \ref{proofstand} is devoted to the proof for the soft sweep from standing variation. It relies on a 
comparison of the population process with a four dimensional dynamical system. 
In Section \ref{section_couplage} we describe a coupling of the population process with two birth and death processes widely 
used in Sections 
\ref{proofstrong} and \ref{proofweak}, respectively devoted to the proofs for the strong and the weak recombination regimes 
of hard sweep.
The proof for the weak regime requires a fine study of the genealogies in a structured 
coalescent process  during the first phase of the selective sweep. We use here some ideas developed in \cite{schweinsberg2005random}.
Finally in the Appendix we state technical results.\\

This work stems from the papers of Champagnat \cite{champagnat2006microscopic} and Schweinsberg and Durrett \cite{schweinsberg2005random}.
 In the sequel, $c$ is used to denote a positive finite constant. Its value can change from 
line to line but it is always independent of the integer $K$ and the positive real number $\eps$. 
The set $\N:=\{1,2,...\}$ denotes the set of positive integers.

\section{Model and main results}\label{model}

We introduce the sets $\mathcal{A}=\{A,a\}$, $\mathcal{B}=\{b_1,b_2\}$, and $\mathcal{E}=\{A,a\}\times \{b_1,b_2\}$ to describe 
the genetic background of individuals. 
The state of the population will be given by the four dimensional Markov process
 $N^{(z,K)}=(N_{\alpha \beta}^{(z,K)}(t), (\alpha, \beta) \in \mathcal{E},t\geq 0)$ where $N_{\alpha \beta}^{(z,K)}(t)$ denotes the number of 
individuals with alleles $(\alpha,\beta)$ at time $t$ when the carrying capacity is $K \in \N$ and the initial state is 
$\lfloor  zK \rfloor $ with $z=(z_{\alpha \beta}, (\alpha, \beta) \in \mathcal{E}) \in \R_+^\mathcal{E}$.
We recall that $b_1$ and $b_2$ are neutral, thus ecological parameters only 
depend on the allele, $A$ or $a$, carried by the individuals at their first locus. There are the following:
\begin{enumerate}
 \item[$\bullet$] For $\alpha \in \mathcal{A}$, $f_\alpha$ and $D_\alpha$ denote the birth rate and the intrinsic death rate of an individual carrying allele 
$\alpha$. 
 \item[$\bullet$] For $(\alpha, \alpha') \in \mathcal{A}^2$, $C_{\alpha,\alpha'}$ represents the competitive pressure felt by an individual carrying
 allele $\alpha$ from an individual carrying allele $\alpha'$.
 \item[$\bullet$] $K \in \N$ is a parameter rescaling the competition between individuals. It can be interpreted as a scale of resources or area available,
 and is related to the concept of carrying capacity, which is the maximum population size that the 
environment can sustain indefinitely. In the sequel $K$ will be large.
 \item[$\bullet$] $r_K$ is the recombination probability per reproductive event. When two individuals with respective genotypes 
$(\alpha,\beta)$ and $(\alpha',\beta')$ in $\mathcal{E}$  give birth, the newborn 
individual, either is a clone of one parent and carries alleles $(\alpha,\beta)$ or $(\alpha',\beta')$ each
with probability $(1-r_K)/2$, 
or has a mixed genotype 
$(\alpha,\beta')$ or $(\alpha',\beta)$ each with probability $r_K/2$. 
\end{enumerate}
 We will use, for every $n=(n_{\alpha\beta}, (\alpha, \beta) \in \mathcal{E}) \in \Z_+^\mathcal{E}$, and $(\alpha,\beta) \in \mathcal{E}$, the notations
\begin{equation*} \label{notR4} n_{\alpha }=n_{\alpha b_1}+n_{\alpha b_2} \quad \text{and} \quad |n|=n_A+n_a.\end{equation*}
Let us now give the transition rates of $N^{(z,K)}$ when $N^{(z,K)}(t)=n\in \Z_+^\mathcal{E}$. 
An individual can die either from a natural death or from competition, whose strength depends on the carrying capacity $K$. 
Thus, the cumulative death rate of individuals 
$\alpha \beta$, with $(\alpha, \beta) \in \mathcal{E}$ is given by:
\begin{equation}\label{death_rate} {d}_{\alpha\beta}^K(n)=\left[ D_\alpha+C_{\alpha,A}n_{A}/K+C_{\alpha,a}n_{a}/K\right]{n_{\alpha\beta}}. \end{equation}
 An individual carrying allele $\alpha \in \mathcal{A}$ produces gametes with rate $f_\alpha$, thus the relative frequencies of gametes available 
for reproduction are 
 $$ p_{\alpha \beta}(n)=f_\alpha n_{\alpha \beta}/(f_A n_{A}+f_a n_{a}), \quad (\alpha, \beta) \in \mathcal{E}.$$
When an individual gives birth, he chooses his mate uniformly among the gametes available. Then the probability of giving birth to an individual 
of a given genotype depends on the parents (the couple $((a,b_2),(a,b_1))$ is not able to generate an individual
 $(A,b_1)$). We detail the computation of the cumulative birth rate of individuals $(A,b_1)$:
\begin{eqnarray*} 
b_{Ab_1}^K(n)&=&f_A n_{Ab_1}[ p_{Ab_1}+{p_{Ab_2}}/{2}+{p_{ab_1}}/{2}+(1-r_K){p_{ab_2}}/{2}]+f_A n_{Ab_2}[ {p_{Ab_1}}/{2}+r_K{p_{ab_1}}/{2}]\\
&&+f_a n_{ab_1}[ {p_{Ab_1}}/{2}+r_K{p_{Ab_2}}/{2} ]+f_a n_{ab_2} (1-r_K){p_{Ab_1}}/{2}\\
&=&f_A n_{Ab_1} + r_K f_A f_a (n_{ab_1}n_{Ab_2}-{n_{Ab_1}}n_{ab_2})/({f_A n_{A}+f_a n_{a}}).
\end{eqnarray*}
If we denote by $\bar{\alpha}$ (resp. $\bar{\beta}$) the complement of $\alpha$ in $\mathcal{A}$ (resp. $\beta$ in $\mathcal{B}$), 
we obtain in the same way the cumulative birth rate of individuals $(\alpha ,\beta)$:
\begin{equation}\label{birth_rate} {b}^K_{\alpha\beta}(n)=f_\alpha n_{\alpha\beta} + r_K f_a f_{A}\frac{n_{\bar{\alpha}\beta}n_{\alpha\bar{\beta}}
-n_{\alpha\beta}n_{\bar{\alpha}\bar{\beta}}}{f_A n_{A}+f_a n_{a}}, \quad (\alpha, \beta) \in \mathcal{E}.\end{equation}
The definitions of death and birth rates in \eqref{death_rate} and \eqref{birth_rate} ensure that the number of jumps is finite on every finite interval, and the population process is well defined.\\

When we focus on the dynamics of traits under selection $A$ and $a$, we get the process \\$ (N^{(z,K)}_A,N^{(z,K)}_a) $, which is 
also a birth and death process with competition. 
It has been studied in \cite{champagnat2006microscopic} and its cumulative death and birth rates, which are direct consequences of \eqref{death_rate} and \eqref{birth_rate}, satisfy for $\alpha \in \mathcal{A}$:
\begin{equation}\label{defdaba}
 {d}_{\alpha}^K(n)=\underset{\beta  \in \mathcal{B}}{\sum} {d}_{\alpha \beta}^K(n)=\Big[ D_\alpha+C_{\alpha,A}\frac{n_{A}}{K}+C_{\alpha,a}\frac{n_{a}}{K}\Big]{n_{\alpha}}, \quad {b}^K_{\alpha}(n)=\underset{\beta  \in \mathcal{B}}{\sum} {b}_{\alpha \beta}^K(n)=f_\alpha n_{\alpha}.
\end{equation}
It is proven in \cite{champagnat2006microscopic} that when $N^{(z,K)}_A$ and $N^{(z,K)}_a$ are of order $K$, 
the rescaled population process $(N^{(z,K)}_A/K,N^{(z,K)}_a/K)$ is well 
approximated by the dynamical system:
\begin{equation} \label{S1}
 \dot{n}_\alpha^{(z)}=(f_\alpha-D_\alpha-C_{\alpha,A}n_A^{(z)}-C_{\alpha,a}n_a^{(z)})n_\alpha^{(z)},\quad n_\alpha^{(z)}(0)=z_\alpha,\quad  \alpha \in \mathcal{A}.
 \end{equation}
More precisely Theorem 3 (b) in \cite{champagnat2006microscopic} states that for every compact
 subset 
$$B \subset (\R_+^{A\times \mathcal{B}}\smallsetminus (0,0)) \times (\R_+^{a\times \mathcal{B}}\smallsetminus (0,0))$$
and finite real number $T$, we have 
for any $\delta>0$,
\begin{equation}\label{result_champa1} \underset{K \to \infty}{\lim}\ \underset{z \in B}{\sup}\  \P\Big(\underset{0\leq t \leq T,\alpha \in \mathcal{A}}
{\sup} |N^{(z,K)}_\alpha(t)/K-{n}^{(z)}_\alpha(t)|\geq \delta \Big)=0.\end{equation}
Moreover, if we assume 
\begin{equation}\label{assumption}
 f_A>D_A, \quad f_a>D_a, \quad \text{and} \quad f_a-D_a>(f_A-D_A).\sup \Big\{{C_{a,A}}/{C_{A,A}}, {C_{a,a}}/{C_{A,a}}\Big\},
\end{equation}
then the dynamical system \eqref{S1} has a unique attracting equilibrium $(0,\bar{n}_a)$ for initial condition $z$ satisfying $z_a>0$, 
and two unstable steady states $(0,0)$ and $(\bar{n}_A,0)$, where
\begin{equation}\label{defnbara1}
 \bar{n}_\alpha=\frac{f_\alpha-D_\alpha}{C_{\alpha,\alpha}}>0,\quad \alpha \in \mathcal{A}.
\end{equation}
Hence, Assumption (\ref{assumption}) avoids the coexistence of alleles $A$ and $a$, and $\bar{n}_\alpha$ is the equilibrium density
of a monomorphic $\alpha$-population per unit of carrying capacity. This implies that when $K$ is large, the size of 
a monomorphic $\alpha$-population stays near $\bar{n}_\alpha K$ for a long time (Theorem 3 (c) in \cite{champagnat2006microscopic}). 
Moreover, if we introduce the invasion fitness $S_{\alpha \bar{\alpha}}$ of a mutant $\alpha$ in a population $\bar{\alpha}$,
\begin{equation}\label{deffitinv1}
 S_{\alpha \bar{\alpha}}=f_\alpha-D_\alpha-C_{\alpha,\bar{\alpha}}\bar{n}_{\bar{\alpha}} ,\quad \alpha \in \mathcal{A},
\end{equation}
it corresponds to the per capita growth 
rate of a mutant $\alpha$ when it appears in a population $\bar{\alpha}$ at its equilibrium density $\bar{n}_{\bar{\alpha}}$. Assumption 
\eqref{assumption} is equivalent to 
\begin{ass}\label{assumption_eq} Ecological parameters satisfy
$$ \bar{n}_A>0, \quad \bar{n}_a>0,\quad \text{and} \quad S_{Aa}<0<S_{aA}.$$
\end{ass}
\noindent Under Assumption \ref{assumption_eq}, with positive probability, the $A$-population becomes extinct and the $a$-population size reaches a vicinity of its 
equilibrium value $\bar{n}_a K$.\\

The case we are interested in is referred in population genetics as soft selection \cite{wallace1975hard}: 
it is both frequency and density dependent. This kind of selection has no influence on the order of the total population size, which has the same order as the carrying capacity $K$. 
However, the factor multiplying the carrying capacity can be modified, as the way the individuals use the resources 
depends on the ecological parameters.
We focus on strong selection coefficient, which are caracterized by $S_{aA}\gg 1/K$. In this case the selection outcompetes
the genetic drift. However we do not need to assume $S_{aA}\ll 1$ to get approximations unlike 
\cite{smith1974hitch,barton1998effect,stephan2006hitchhiking}.
To study the genealogy of the selected allele when the selection coefficient is weak ($S_{aA}K$ moderate or small)
we refer to the approach of Neuhauser and Krone \cite{neuhauser1997genealogy}.\\

Let us now present the main results of this paper. We introduce the extinction time of the $A$-population, and the fixation event 
of the $a$-population. For $(z,K) \in \R_+^{\mathcal{E}}\times \N$:
\begin{equation}\label{defText}T^{(z,K)}_{\text{ext}}:=\inf \Big\{ t \geq 0, N_A^{(z,K)}(t)=0  \Big\},\quad \text{and} \quad 
\text{Fix}^{(z,K)}:=\Big\{T^{(z,K)}_{\text{ext}}<\infty,  N_a^{(z,K)}(T^{(z,K)}_{\text{ext}})>0\Big\}  .\end{equation}
We are interested in the neutral allele proportions. We thus define for $t \geq 0$,
\begin{equation}\label{def_proportion} P_{\alpha, \beta}^{(z,K)}(t) = \frac{N^{(z,K)}_{\alpha \beta}(t)}{N^{(z,K)}_{\alpha}(t)}, \quad (\alpha,\beta) \in \mathcal{E}, K\in \N,  z \in \R_+^{\mathcal{E}},\end{equation}
 the proportion of alleles $\beta$ in the $\alpha$-population at time $t$.
More precisely, we are interested in these proportions at the end of the sweep, that is at time $T^{(z,K)}_{\text{ext}}$ when the last $A$-individual 
dies. We then introduce the neutral proportion at this time:
\begin{equation}\label{def_proportiontext} \mathcal{P}_{a, b_1}^{(z,K)}=P_{a, b_1}^{(z,K)}(T^{(z,K)}_{\text{ext}}). \end{equation}

We first focus on soft selective sweeps from standing variation. We assume that the alleles $A$ and $a$ were neutral and coexisted in a population with large carrying capacity $K$. 
At time $0$, an environmental change makes the allele $a$ favorable (in the sense of Assumption \ref{assumption_eq}). Before stating the result, 
let us introduce the function $F$, defined for every $(z,r,t) \in (\R_+^{\mathcal{E}})^* \times [0,1] \times \R_+$ by
\begin{equation}\label{defF} 
 F(z,r,t)=\int_0^t \frac{rf_Af_an_A^{(z)}(s)}{f_An_A^{(z)}(s)+f_an_a^{(z)}(s)}\exp\Big( -rf_Af_a\int_0^s \frac{n_A^{(z)}(u)+n_a^{(z)}(u)}{f_An_A^{(z)}(u)+f_an_a^{(z)}(u)}du \Big)ds,
\end{equation}
where $(n_A^{(z)},n_a^{(z)})$ is the solution of the dynamical system \eqref{S1}. 
We notice that $F:t \in \R^+ \mapsto F(z,r,t)$ is non-negative and non-decreasing. Moreover, if we introduce the function
$$h: (z,r,t) \in (\R_+^{\mathcal{E}})^* \times [0,1] \times \R_+ \mapsto rf_Af_a\int_0^t {n_A^{(z)}(s)}/({f_An_A^{(z)}(s)+f_an_a^{(z)}(s)})ds$$
non-decreasing 
in time, then 
$$ 0\leq F(z,r,t)\leq \int_0^t \partial_s h(z,r,s)e^{-h(z,r,s)}ds=e^{-h(z,r,0)}-e^{-h(z,r,t)}\leq 1. $$
Thus $F(z,r,t)$ has a limit in $[0,1]$ when $t$ goes to infinity and we can define 
\begin{equation}\label{limF} F(z,r):=\lim_{t \to \infty}F(z,r,t) \in [0,1] .\end{equation}
Noticing that for every $r \in [0,1]$ and $t \geq 0$, 
$$ 0\leq F(z,r)- F(z,r,t)\leq \int_t^\infty \frac{f_Af_an_A^{(z)}(s)}{f_An_A^{(z)}(s)+f_an_a^{(z)}(s)}ds \underset{t\to \infty}{\to}0,$$
we get that the convergence of $(F(z,r,t), t \geq 0)$ is uniform for $r \in [0,1]$.

In the case of a soft sweep from standing variation, the selected allele gets to fixation with high probability. More precisely, 
it is proven in \cite{champagnat2006microscopic} that under Assumption \ref{assumption_eq},
\begin{equation}\label{convfixcas1} \underset{K \to \infty}{\lim}\P(\text{Fix}^{(z,K)})= 1, 
\quad \forall z \in \R_+^{A\times \mathcal{B}} \times (\R_+^{a\times \mathcal{B}}\setminus (0,0)).\end{equation}
\noindent Then recalling \eqref{def_proportiontext} we get the following result whose proof is deferred to Section \ref{proofstand}:

\begin{theo} \label{main_result2}
Let $z$ be in $\R_+^{A\times \mathcal{B}} \times (\R_+^{a\times \mathcal{B}}\setminus (0,0))$ and Assumption \ref{assumption_eq} hold. 
Then on the fixation event $\textnormal{Fix}^{(z,K)}$, the proportion of alleles $b_1$ when the $A$-population becomes extinct
 (time $T_{\text{ext}}^{(z,K)}$) converges in probability:
$$ \underset{K \to \infty}{\lim}\P \Big(\mathbf{1}_{\textnormal{Fix}^{(z,K)}}\Big| \mathcal{P}_{a,b_1}^{(z,K)}-\Big[\frac{z_{Ab_1}}{z_A}F(z,r_K) +\frac{z_{ab_1}}{z_a}(1-F(z,r_K))\Big]\Big| >\eps\Big) = 0
, \quad \forall \eps>0.   $$
\end{theo}
The neutral proportion at the end of a soft sweep from standing variation is thus a weighted mean of initial proportions in populations $A$ and $a$. 
In particular, a soft 
sweep from standing variation is 
responsible for a diminution of the number of neutral alleles with very low or very high proportions in the population, as remarked in 
\cite{prezeworski2005signature}. We notice that the 
weight $F(z,r_K)$ does not depend on the initial neutral proportions. It
only depends on $r_K$ and on the dynamical system \eqref{S1} with initial condition $(z_A,z_a)$. 
The proof consists in comparing the population process with the four dimensional dynamical system,
\begin{equation} \label{syst_dyn}  
 \dot{n}_{\alpha \beta}^{(z,K)} =  \left( f_\alpha -D_\alpha-C_{\alpha,A}{n}^{(z,K)}_{A}-C_{\alpha,a}{n}^{(z,K)}_{a}\right){n}^{(z,K)}_{\alpha  \beta}+
  rf_A f_a\frac{ {n}^{(z,K)}_{\bar{\alpha} \beta}{n}^{(z,K)}_{\alpha \bar{\beta}}-{n}^{(z,K)}_{\alpha \beta}
{n}^{(z,K)}_{\bar{\alpha}\bar{\beta}}}{f_A {n}^{(z,K)}_{A}+f_a {n}^{(z,K)}_{a}}, \quad (\alpha,\beta)\in \mathcal{E},
 \end{equation}
with initial condition $n^{(z,K)}(0)=z \in \R_+^\mathcal{E}$. Then by a change of variables, we can study the dynamical system \eqref{syst_dyn} 
and prove that 
\begin{equation*}\frac{n^{(z,K)}_{a,b_1}(\infty)}{n^{(z,K)}_{a}(\infty)}= \frac{z_{Ab_1}}{z_A} F(z,r_K)
+\frac{z_{ab_1}}{z_a}(1-F(z,r_K)),\end{equation*} which leads to the result.\\

Now we focus on hard selective sweeps: a mutant $a$ appears in a large population and gets to fixation. We assume that the mutant appears when the 
$A$-population is at ecological equilibrium, and carries the neutral allele $b_1$. In other words, recalling Definition \eqref{defnbara1}, we assume:

\begin{ass}\label{defrK}
 There exists $z_{Ab_1} \in ]0,\bar{n}_A[$ such that $N^{(z^{(K)},K)}(0)=\lfloor z^{(K)}K\rfloor$ with 
$$ z^{(K)}=(z_{Ab_1} ,\bar{n}_A-z_{Ab_1},K^{-1}, 0). $$
\end{ass}
\noindent In this case, the selected allele gets to fixation with positive probability. 
More precisely, it is proven {in \cite{champagnat2006microscopic} that} under Assumptions \ref{assumption_eq} and \ref{defrK},
\begin{equation} \label{proba_fix}\underset{K \to \infty}{\lim}\P\Big(\text{Fix}^{(z^{(K)},K)}\Big)= \frac{S_{aA}}{f_a}.\end{equation}
In the case of a strong selective sweep we will distinguish two different recombination regimes:

\begin{ass}\label{condstrong} Strong recombination
 $$\lim_{K \to \infty}\  r_K\log K=\infty .$$
\end{ass}

\begin{ass} \label{condweak} Weak recombination
 $$\limsup_{K \to \infty}\  r_K\log K<\infty .$$
\end{ass}
\noindent Recall \eqref{def_proportiontext} and introduce the real number
\begin{equation}
 \label{defrhoK} \rho_K:= 1-\exp \Big( -\frac{f_ar_K\log K}{S_{aA}} \Big).
\end{equation}
Then we have the following results whose proofs are deferred 
to Sections \ref{proofstrong} and \ref{proofweak}:

\begin{theo} \label{main_result}
Suppose that Assumptions \ref{assumption_eq} and \ref{defrK} hold. Then on the fixation event $\textnormal{Fix}^{(z^{(K)},K)}$ and under Assumption \ref{condstrong} 
or \ref{condweak}, the proportion of alleles $b_1$ when the $A$-population becomes extinct (time $T_{\text{ext}}^{(z^{(K)},K)}$) converges in probability. More precisely, 
if Assumption \ref{condstrong} holds,
$$ \lim_{K \to \infty}\P \Big(\mathbf{1}_{\textnormal{Fix}^{(z^{(K)},K)}}\Big|\mathcal{P}_{a,b_1}^{(z^{(K)},K)}
 -\frac{z_{Ab_1}}{z_A}\Big|>\eps\Big) = 0
, \quad \forall \eps>0,  $$
and if Assumption \ref{condweak}  holds, 
$$ \lim_{K \to \infty}\P \Big(\mathbf{1}_{\textnormal{Fix}^{(z^{(K)},K)}}\Big|\mathcal{P}_{a,b_1}^{(z^{(K)},K)} 
-\Big[ (1-\rho_K)+ \rho_K\frac{z_{Ab_1}}{z_A}\Big]\Big|>\eps\Big)= 0
, \quad \forall \eps>0.   $$
\end{theo}

As stated in \cite{champagnat2006microscopic}, the selective sweep has a duration of order $\log K$. Thus, when $r_K \log K$ is large, a lot 
of recombinations occur during the sweep, and the neutral alleles are constantly exchanged by the 
populations $A$ and $a$. Hence in the strong recombination case, the sweep does not modifiy the neutral 
allele proportion. On the contrary,  when $r_K$ 
is of order $1/\log K$ the number of recombinations undergone by a given lineage does not go to infinity, and the frequency of the neutral 
allele $b_1$ carried by the first mutant $a$ increases. 
More precisely, we will show that the probability for a neutral lineage to undergo a recombination and be descended from an individual 
of type $A$ alive at the beginning of the sweep is close to $\rho_K$. 
Then to know 
the probability for such an allele to be
a $b_1$ or a $b_2$, we have to approximate the proportion of alleles $b_1$ in the $A$-population
 when the recombination occurs.
We will prove that this proportion stays close to the initial one $z_{Ab_1}/z_A$ during the first phase.
With probability $1-\rho_K$, a neutral allele originates from the first mutant. 
In this case it is necessarily a $b_1$. This gives the result for the weak recombination regime.
In fact the probability for a neutral lineage to undergo no recombination during the first phase is quite intuitive: broadly speaking, the probability to have 
no recombination at a birth 
event is $1-r_K$, the birth rate is $f_a$ and the duration of the first phase is $\log K/S_{aA}$. 
Hence as $r_K$ is small for large $K$, $1-r_K \sim \exp(-r_K)$ and the probability to undergo no recombination is approximately
$$ (1-r_K)^{f_a\log K/S_{aA}}\sim \exp(-r_K{f_a\log K/S_{aA}})=1-\rho_K. $$

\begin{rem}
 The limits in the two regimes are consistent in the sense that 
$$\underset{r_K \log K \to \infty}{\lim} \rho_K=1. $$
Moreover, let us notice that we can easily extend the results of Theorems \ref{main_result2} and \ref{main_result} to a finite number of possible alleles $b_1$, $b_2$, ..., $b_i$ on the neutral locus.
\end{rem}

\begin{rem}
As it will appear in the proofs (see Sections \ref{proofstrong} and \ref{proofweak}), the final neutral proportion
in the $a$ population is already 
 reached at the end of the first phase. In particular, the results are still valid if the sweep is not complete but the allele 
 $a$ only reaches a fraction $0<p<1$ of the population at the end of the sweep. The fact that the final neutral proportion is 
 mostly determined by the beginning of the sweep has already been noticed by Coop and Ralph in \cite{coop2012patterns}.
\end{rem}

\section{A semi-martingale decomposition}\label{prel_results}

The expression of birth rate in \eqref{birth_rate} shows that the effect of recombination depends on the recombination probability 
$r_K$ but also on the population state via the term $n_{\bar{\alpha}\beta}n_{\alpha\bar{\beta}}-n_{\alpha\beta}n_{\bar{\alpha}\bar{\beta}}$. 
Proposition \ref{mart_prop} states a semi-martingale representation 
of the neutral allele proportions and makes this interplay more precise.

\begin{pro}\label{mart_prop}
Let $(\alpha,z, K)$ be in $\mathcal{A}\times (\R_+^{\mathcal{E}})^*\times \N$. The process $(P_{\alpha,b_1}^{(z,K)}(t),t\geq 0)$ defined in 
\eqref{def_proportion} is a semi-martingale and we have the following decomposition:
\begin{multline} \label{defM} P_{\alpha,b_1}^{(z,K)}(t)=P_{\alpha,b_1}^{(z,K)}(0)+ M^{(z,K)}_\alpha(t)\\+
 r_Kf_A f_a \int_0^{t } {\mathbf{1}_{\{N_\alpha(s)\geq 1\}}}  \frac{N_{\bar{\alpha}b_1}^{(z,K)}(s)N^{(z,K)}_{\alpha b_2}(s)-N^{(z,K)}_{\alpha b_1}(s)N^{(z,K)}_{\bar{\alpha}b_2}(s)}
{(N^{(z,K)}_{\alpha}(s)+1)(f_A N^{(z,K)}_{A}(s)+f_a N^{(z,K)}_{a}(s))}ds , \end{multline}
where the process $(M_\alpha^{(z,K)}(t),t\geq 0)$ is a martingale bounded on every interval $[0,t]$ whose quadratic variation is given by \eqref{crochet}. 
\end{pro}

To lighten the presentation in remarks and proofs we shall mostly write $N$ instead of $N^{(z,K)}$.

\begin{rem} \label{remLD}
The process $N_{ab_2}N_{Ab_1}-N_{ab_1}N_{Ab_2}$ will play a major role in the dynamics of neutral proportions. Indeed it is a 
measure of the neutral proportion disequilibrium between the $A$ and $a$-populations as it satisfies:
\begin{equation}\label{equaldiffprop} N_{A}N_{a}(P_{A,b_1}-P_{a,b_1})={N_{ab_2}N_{Ab_1}-N_{ab_1}N_{Ab_2}}. \end{equation}
This quantity is linked with the linkage disequilibrium of the population, which is the occurrence of some allele combinations more or less often 
than would be expected from a random formation of haplotypes (see \cite{durrett2008probability} Section 3.3 for an introduction to this 
notion or \cite{mcvean2007structure} for a study of its structure around a sweep).
\end{rem}

\begin{rem}
 By taking the expectations in Proposition \ref{mart_prop} we can make a comparison with the results of Ohta and Kimura \cite{ohta1975effect}.
In their work the population size is infinite and the proportion of favorable allele $(y_t, t \geq 0)$ evolves as a deterministic logistic curve: 
$$ \frac{dy_t}{dt}=sy_t(1-y_t) .$$
Moreover, $x_1$ and $x_2$ denote the neutral proportions of a given allele in the selected and non selected 
populations respectively, and are modeled 
by a diffusion. 
By making the analogies
$$ N_e(t)=N_A(t)+N_a(t), \quad y_t=\frac{N_a(t)}{N_A(t)+N_a(t)}, \quad x_1(t)=\frac{N_{ab_1}(t)}{N_a(t)}, \quad x_2(t)=\frac{N_{Ab_1}(t)}{N_A(t)} ,$$
where $N_e$ is the effective population size, the results of \cite{ohta1975effect} can be written
$$ \frac{d\E[P_{\alpha,b_1}(t)]}{dt}=r \frac{\E[N_{\bar{\alpha}b_1}(t)N_{\alpha b_2}(t)-N_{\alpha b_1}(t)N_{\bar{\alpha}b_2}(t)]}
{N_{\alpha}(t)(N_{A}(t)+ N_{a}(t))}, $$
and
\begin{eqnarray*} \frac{d\E[P^2_{\alpha,b_1}(t)]}{dt}&=& \frac{\E[P_{\alpha b_1}(t)(1-P_{\alpha b_2}(t))]}{2N_\alpha(t)}+
 2r \frac{\E[N_{\alpha b_1}(t)(N_{\bar{\alpha}b_1}(t)N_{\alpha b_2}(t)-N_{\alpha b_1}(t)N_{\bar{\alpha}b_2}(t))]}
{N_{\alpha}^2(t)(N_{A}(t)+ N_{a}(t))}. \end{eqnarray*}
Hence the dynamics of the first moments are very similar to these that we obtain when we take equal birth rates $f_A=f_a$ and a recombination $r_K=r/f_a$.
In contrast, the second moments of neutral proportions are very different in the two models.
\end{rem}

\begin{proof}[Proof of Proposition \ref{mart_prop}]
In the vein of Fournier and M\'el\'eard \cite{fournier2004microscopic} we represent the population process in terms 
of Poisson measure. Let $Q(ds,d\theta)$ 
be a Poisson random measure on $\R_+^2$ with intensity $ds d\theta$, and $(e_{\alpha \beta}, (\alpha, \beta)\in \mathcal{E} )$ the canonical 
basis of $\R^\mathcal{E}$.
According to \eqref{defdaba} a jump occurs at rate 
$$\sum_{(\alpha,\beta)\in \mathcal{E}}(b_{\alpha \beta}^K(N)+d_{\alpha \beta}^K(N))=f_aN_a+d_a^K(N)+f_AN_A+d_A^K(N).$$
We decompose on possible jumps that may occur: births and deaths for $a$-individuals and births and deaths for $A$-individuals. 
Itô's formula with jumps (see \cite{ikeda1989stochastic} p. 66) yields for 
every function $h$ measurable and bounded on $\R_+^{\mathcal{E}}$:
\begin{eqnarray}\label{defN}
 h(N(t))&=&h(N(0))+ \int_0^t\int_{R_+} \Big\{ \underset{\alpha \in \mathcal{A}}{ \sum}   \Big(
h(N({s^-})+e_{\alpha b_1})\mathbf{1}_{0<\theta-\mathbf{1}_{\alpha=A}(f_{a}N_a(s^-)+d_{a}^K(N({s^-}))\leq b^K_{\alpha b_1}(N({s^-}))}\nonumber \\
&&\hspace{1cm}+h(N({s^-})+e_{\alpha b_2})\mathbf{1}_{b^K_{\alpha b_1}(N({s^-}))<\theta-\mathbf{1}_{\alpha=A}(f_{a}N_a(s^-)+d_{a}^K(N({s^-}))\leq f_\alpha N_\alpha({s^-})}\nonumber \\
&&\hspace{1cm}+h(N({s^-})-e_{\alpha b_1})\mathbf{1}_{0<\theta-f_\alpha N_\alpha(s^-)- \mathbf{1}_{\alpha=A}(f_{a}N_a(s^-)+d_{a}^K(N({s^-}))\leq d_{\alpha b_1}^K(N({s^-}))}\nonumber \\
&&\hspace{1cm}+h(N({s^-})-e_{\alpha b_2})\mathbf{1}_{d^K_{\alpha b_1}(N({s^-}))<\theta-f_\alpha N_\alpha(s^-)- \mathbf{1}_{\alpha=A}(f_{a}N_a(s^-)+d_{a}^K(N({s^-}))\leq d_{\alpha }^K(N({s^-}))}\Big)\nonumber\\
&&\hspace{1cm}- h(N({s^-}))\mathbf{1}_{ \theta \leq f_{a}N_a(s^-)+d_a^K(N({s^-}))+f_{A}N_A(s^-) +d_{A}^K(N({s^-}))}\Big\} Q(ds,d\theta).
\end{eqnarray}
Let us introduce the functions 
$\mu^\alpha_{K}$ defined for $\alpha \in \mathcal{A}$ and $(s,\theta)$ in $ \R_+ \times \R_+$ by, 
\begin{eqnarray} \label{defmua}\mu^\alpha_{K}(N,s,\theta)&=&\frac{\mathbf{1}_{N_\alpha(s)\geq 1}N_{\alpha b_2}(s)}{(N_{\alpha }(s)+1)N_{\alpha }(s)}
\mathbf{1}_{0<\theta-\mathbf{1}_{\alpha=A}(f_{a}N_a(s)+d_{a}^K(N({s}))\leq b^K_{\alpha b_1}(N({s}))}\\
&&- \frac{\mathbf{1}_{N_\alpha(s)\geq 1}N_{\alpha b_1}(s)}{(N_{\alpha}(s)+1)N_{\alpha}(s)}
\mathbf{1}_{b^K_{\alpha b_1}(N({s}))<\theta-\mathbf{1}_{\alpha=A}(f_{a}N_a(s)+d_{a}^K(N({s}))\leq f_\alpha N_\alpha({s})}\nonumber\\
&&- \frac{\mathbf{1}_{N_\alpha(s)\geq 2}N_{\alpha b_2}(s)}{(N_{\alpha}(s)-1)N_{\alpha}(s)} 
 \mathbf{1}_{0<\theta-f_\alpha N_\alpha(s)- \mathbf{1}_{\alpha=A}(f_{a}N_a(s)+d_{a}^K(N({s}))\leq d_{\alpha b_1}^K(N({s}))}\nonumber\\
&&+\frac{\mathbf{1}_{N_\alpha(s)\geq 2}N_{\alpha b_1}(s)}{(N_{\alpha}(s)-1)N_{\alpha}(s)}
\mathbf{1}_{d^K_{\alpha b_1}(N({s}))<\theta-f_\alpha N_\alpha(s)- \mathbf{1}_{\alpha=A}(f_{a}N_a(s)+d_{a}^K(N({s}))\leq d_{\alpha }^K(N({s}))}.\nonumber
\end{eqnarray}
Then we can represent the neutral allele proportions $P_{\alpha,b_1}$ as,
\begin{equation} \label{ecripbiamu} P_{\alpha,b_1}(t)=P_{\alpha,b_1}(0)+\int_0^t \int_{0}^\infty \mu^\alpha_{K}(N,s^-,\theta)Q(ds,d\theta),\quad  t\geq 0. \end{equation}
A direct calculation gives
$$ \int_0^\infty \mu^\alpha_{K}(N,s,\theta)d\theta=r_Kf_Af_a {\mathbf{1}_{\{N_\alpha(s)\geq 1\}}}\frac{N_{\bar{\alpha}b_1}(s)N_{\alpha b_2}(s)-N_{\alpha b_1}(s)N_{\bar{\alpha}b_2}(s)}{(N_{\alpha}(s)+1)(f_A N_{A}(s)+f_a N_{a}(s))} .$$
Thus if we introduce the compensated Poisson measure $\tilde{Q}(ds,d\theta):={Q}(ds,d\theta)-dsd\theta$, then
\begin{multline*} M_\alpha(t):= \int_0^t \int_0^\infty \mu^\alpha_{K} (N,s^-,\theta)\tilde{Q}(ds,d\theta)
 \\= P_{\alpha,b_1}(t)-P_{\alpha,b_1}(0)-r_Kf_A f_a   
\int_0^{t }{\mathbf{1}_{\{N_\alpha(s)\geq 1\}}}  \frac{N_{\bar{\alpha}b_1}(s)N_{\alpha b_2}(s)-N_{\alpha b_1}(s)N_{\bar{\alpha}b_2}(s)}{(N_{\alpha}(s)+1)(f_A N_{A}(s)+f_a N_{a}(s))}ds 
\end{multline*}
is a local martingale. By construction the process $P_{\alpha,b_1}$ has values in $[0,1]$ and as $r_K\leq 1$,
\begin{equation}\label{boundmart} \underset{s\leq t}{\sup}\ \Big|  r_Kf_A f_a \int_0^{s}  {\mathbf{1}_{\{N_\alpha\geq 1\}}}
\frac{ N_{\bar{\alpha}b_1}N_{\alpha b_2}-N_{\alpha b_1}N_{\bar{\alpha}b_2}}{(N_{\alpha}+1)(f_A N_{A}+f_a N_{a})} \Big|\leq r_K f_\alpha t \leq  f_\alpha t ,\quad t\geq 0.\end{equation}
Thus $M_\alpha$ is a square integrable pure jump martingale bounded on every finite interval with quadratic variation
\begin{eqnarray}\label{crochet}  \langle M_\alpha \rangle_{t} &=& 
\int_0^{t} \int_0^\infty\Big( \mu^\alpha_{K}(N,s,\theta)\Big)^2dsd\theta \nonumber\\
&=& \int_0^{t}\Big\{ P_{\alpha,b_1}(1-P_{\alpha,b_1}) \Big[\Big(D_\alpha+\frac{C_{\alpha, A}}{K}N_{A}+\frac{C_{\alpha,a}}{K}N_{a}\Big) \frac{\mathbf{1}_{N_\alpha\geq 2}N_\alpha}{(N_\alpha-1)^2} \nonumber \\
&&\hspace{.5cm}  + \frac{f_\alpha N_\alpha}{(N_{\alpha}+1)^2}\Big] +  
 r_Kf_A f_{a}{\mathbf{1}_{\{N_\alpha\geq 1\}}}\frac{(N_{\bar{\alpha}b_1}N_{\alpha b_2}-N_{\alpha b_1}N_{\bar{\alpha}b_2}) (1-2P_{\alpha,b_1})}{(N_{\alpha}+1)^2(f_A N_{A}+f_{a} N_{a})}\Big\} .
 \end{eqnarray}
This ends the proof of Proposition \ref{mart_prop}. \end{proof}

\begin{rem}
By definition of the functions $\mu^\alpha_K$ in \eqref{defmua} we have for all $(s,\theta)$ 
in $ \R_+ \times \R_+$,
\begin{equation}\label{muAmua0}
 \mu_K^A(N,s,\theta)\mu_K^a(N,s,\theta)=0.
\end{equation}
\end{rem}

Lemma  \ref{lemmualpha} states properties of the quadratic variation $\langle M_\alpha \rangle$ widely used in the forthcoming proofs. 
We introduce a compact interval containing 
the equilibrium size of the $A$-population,
\begin{equation} \label{compact1}I_\eps^K:= \Big[K\Big(\bar{n}_A-2\eps \frac{C_{A,a}}{C_{A,A}}\Big),K\Big(\bar{n}_A+2\eps \frac{C_{A,a}}{C_{A,A}}\Big)\Big]\cap \N, \end{equation}
and the stopping times $T^K_\eps$ and $S^K_\eps$, which denote respectively the hitting time of $\lfloor\eps K \rfloor$ by the mutant population and the exit time of $I_\eps^K$ by the resident population,
\begin{equation} \label{TKTKeps1} T^K_\eps := \inf \Big\{ t \geq 0, N^K_a(t)= \lfloor \eps K \rfloor \Big\},\quad S^K_\eps := \inf \Big\{ t \geq 0, N^K_A(t)\notin I_\eps^K \Big\}.  \end{equation}
Finally we introduce a constant depending on $\alpha \in \mathcal{A}$ and $\nu \in \R_+^*$,
\begin{equation}\label{defcalphanu} C(\alpha,v):=4D_\alpha+2f_\alpha+4(C_{\alpha,A}+C_{\alpha,a})\nu. \end{equation}

\begin{lem}\label{lemmualpha}
For $v<\infty$ and $t \geq 0$ such that $(N_A^{(z,K)}(t),N_a^{(z,K)}(t))\in 
[0,vK]^2$, 
\begin{equation}\label{crocheten1K1}
 {\frac{d}{dt}\langle M_\alpha^{(z,K)} \rangle_t }=\int_0^\infty\Big( \mu^\alpha_{K}(N^{(z,K)},t,\theta)\Big)^2d\theta\leq {C(\alpha,v)}\frac{\mathbf{1}_{N_\alpha(t)\geq 1}}{N_\alpha(t)},
 \quad \alpha \in \mathcal{A}.
\end{equation}
Moreover, under Assumptions \ref{assumption_eq} and \ref{defrK}, there exist $k_0 \in \N $, $\eps_0>0$ and a pure jump martingale $\bar{M}$ such that for 
$\eps\leq \eps_0$ and $t\geq 0$, 
\begin{equation} \label{dcrocheta} 
e^{\frac{S_{aA}}{2(k_0+1)}t\wedge {T}^K_\eps \wedge S^K_\eps}
 \int_0^\infty\Big( \mu^a_{K}(N^{(z^{(K)},K)},t\wedge {T}^K_\eps \wedge S^K_\eps,\theta)\Big)^2d
\theta \leq (k_0+1)C(a,2\bar{n}_A)  \bar{M}_{t\wedge {T}^K_\eps \wedge S^K_\eps},
\end{equation}
and 
\begin{equation}\label{tildemart}
 \E\Big[\bar{M}_{t\wedge {T}^K_\eps \wedge S^K_\eps} \Big]\leq \frac{1}{k_0+1}.
\end{equation}

\end{lem}

\begin{proof}
 Equation \eqref{crocheten1K1} is a direct consequence of \eqref{crochet}. To prove \eqref{dcrocheta} and \eqref{tildemart}, let us first notice that
 according to Assumption \ref{assumption_eq}, there exists $k_0 \in \N$ such that for $\eps$ small enough and $k \in \Z_+$,
 $$\frac{f_a(k_0+k-1)-(D_a+C_{a,A}\bar{n}_A+\eps(C_{a,a}+2C_{A,a}C_{a,A}/C_{A,A}))(k_0+k+1)}{k_0+k-1}\geq \frac{S_{aA}}{2}.$$
This implies in particular that for every $t <  {T}^K_\eps \wedge S^K_\eps$, 
\begin{equation}\label{k0}  \frac{
f_a N_a(t)(N_a(t)+k_0-1)-d_a^K(N(t))(N_a(t)+k_0+1)}{(N_a(t)+k_0-1)(N_a(t)+k_0+1)}\geq \frac{S_{aA}N_a(t)}{2(N_a(t)+k_0+1)}\geq \frac{S_{aA}\mathbf{1}_{N_a(t)\geq 1}}{2(k_0+1)},\end{equation} 
where the death rate $d_a^K$ has been defined in \eqref{defdaba}. For sake of simplicity let us introduce the process $X$ defined as follows:
$$ {X(t)=\frac{1}{N_a({t})+k_0}\exp\Big(\frac{S_{aA}t}{2(k_0+1)}\Big),\quad \forall t\geq 0.} $$
Applying Itô's formula with jumps we get for every $t\geq 0$:
\begin{multline} \label{itotilde}
 X({t\wedge {T}^K_\eps \wedge S^K_\eps})=\bar{M}(t\wedge {T}^K_\eps \wedge S^K_\eps)+\\
\int_0^{t\wedge {T}^K_\eps \wedge S^K_\eps}\Big(\frac{S_{aA}}{2(k_0+1)}- \frac{f_a N_a(s)(N_a(s)+k_0-1)-d_a^K(N(s))(N_a(s)+k_0+1)
}{(N_a(s)+k_0-1)(N_a(s)+k_0+1)}\Big)X(s)ds,
\end{multline}
where the martingale $\bar{M}$ has the following expression:
\begin{multline} \label{exprtildeM} \bar{M}(t)=\frac{1}{k_0+1}+\int_0^t\int_{\R_+}\tilde{Q}(ds,d\theta)\exp\Big(\frac{S_{aA}s}{2(k_0+1)}\Big)\\
\Big[ 
 \frac{ \mathbf{1}_{\theta \leq f_aN_a(s^-)} }{N_a(s^-)+k_0+1}
+\frac{ \mathbf{1}_{ f_aN_a(s^-)<\theta \leq f_aN_a(s^-)+d_a(N(s^-))} }{N_a(s^-)+k_0-1}
-\frac{\mathbf{1}_{\theta \leq f_aN_a(s^-)+d_a(N(s^-))} }{N_a(s^-)+k_0} 
\Big] .\end{multline}
 Thanks to \eqref{k0} the integral in \eqref{itotilde} is nonpositive. Moreover, according to \eqref{crocheten1K1}, for $t \leq {T}^K_\eps \wedge 
S^K_\eps$, 
as $2\eps{C_{A,a}}/{C_{A,A}}\leq \bar{n}_A$ for $\eps$ small enough,
\begin{multline}  \int_0^\infty\Big( \mu^a_{K}(N^{(z^{(K)},K)},t,\theta)\Big)^2d
\theta \leq C(a,2\bar{n}_A)\frac{\mathbf{1}_{N_a(t)\geq 1}}{N_a(t)}\\
\leq (k_0+1)C(a,2\bar{n}_A)X(t)\exp\Big(-\frac{S_{aA}t}{2(k_0+1)}\Big) ,\end{multline}
which ends the proof.\end{proof}

\section{Proof of Theorem \ref{main_result2}}\label{proofstand}
In this section we suppose that Assumption \ref{assumption_eq} holds. For $\eps\leq C_{a,a}/C_{a,A} \wedge 2|S_{Aa}|/C_{A,a}$ and $z$ 
in $\R_+^{A\times \mathcal{B}}\times (\R_+^{a\times \mathcal{B}} \setminus (0,0) )$ we introduce a deterministic time $t_{\eps}(z)$ after which the solution $(n_A^{(z)},n_a^{(z)})$ of the dynamical system \eqref{S1} is close to the stable equilibrium $(0,\bar{n}_a)$:
\begin{equation}\label{deftepsz1} t_{\eps}(z):=\inf \big\{ s \geq 0,\forall t \geq s, ({n}_A^{(z)}(t),{n}_a^{(z)}(t))\in 
[0,\eps^2/2]\times[\bar{n}_a-\eps/2,\infty) \big\}. \end{equation}
Once $(n_A^{(z)},n_a^{(z)})$ has reached the set $[0,\eps^2/2]\times[\bar{n}_a-\eps/2,\infty)$ it 
never escapes from it. Moreover, according to Assumption \ref{assumption_eq} on 
the stable equilibrium, $t_\eps(z)$ is finite.\\

  First we compare the population process with the four dimensional dynamical system \eqref{syst_dyn} on the
 time interval $[0, t_\eps(z)]$. Then we study this dynamical system and get an approximation of
the neutral proportions at time $t_\eps(z)$. Finally, we state that
during the A-population extinction period, this proportion stays nearly constant.

\subsection{Comparison with a four dimensional dynamical system}

Recall that $n^{(z,K)}=(n^{(z,K)}_{\alpha \beta},(\alpha,\beta)\in \mathcal{E})$ is the solution of the dynamical system \eqref{syst_dyn}  
with initial condition $z$. Then we have the following comparison result:

\begin{lem}\label{lemapprox}
 Let $z$ be in $\R_+^\mathcal{E}$ and $\eps$ be in $\R_+^*$. Then
\begin{equation}\label{EK2}
\underset{K \to \infty}{\lim}\  \sup_{s\leq t_\eps(z)}\ \|{N}^{(z,K)}(s)/K-n^{(z,K)}(s)  \|=0 \quad \text{in probability}
\end{equation}
where $\|. \|$ denotes the $L^1$-Norm on $\R^\mathcal{E}$.
\end{lem}

\begin{proof}
The proof relies on a slight modification of Theorem 2.1 p. 456 in Ethier and Kurtz \cite{EK}.
According to \eqref{death_rate} and \eqref{birth_rate}, the rescaled birth and death rates
\begin{equation}\label{deftildeb} \tilde{b}_{\alpha \beta}^K(n)=\frac{1}{K}b_{\alpha \beta}^K(Kn)=f_\alpha n_{\alpha\beta} + r_K f_a f_{A}\frac{n_{\bar{\alpha}\beta}n_{\alpha\bar{\beta}}
-n_{\alpha\beta}n_{\bar{\alpha}\bar{\beta}}}{f_A n_{A}+f_a n_{a}}, \quad (\alpha, \beta) \in \mathcal{E}, n \in N^\mathcal{E}, \end{equation}
and
\begin{equation}\label{deftilded} \tilde{d}_{\alpha \beta}(n)=\frac{1}{K}d_{\alpha \beta}^K(Kn)=\left[ D_\alpha+C_{\alpha,A}n_{A}+C_{\alpha,a}n_{a}\right]{n_{\alpha\beta}}, 
\quad (\alpha, \beta) \in \mathcal{E}, n \in N^\mathcal{E}, \end{equation}
are Lipschitz and bounded on every compact subset of $ \N^\mathcal{E}$. The only difference with \cite{EK} is that $ \tilde{b}_{\alpha \beta}^K$ 
depends on $K$ via the term $r_K$.
Let $(Y_i^{(\alpha\beta)},i \in \{1,2\}, (\alpha,\beta)\in \mathcal{E})$ be eight independent standard Poisson processes.
From the representation of the population process $N^{(z,K)}$ in \eqref{defN} we see that the process $(\bar{N}^{(z,K)}(t), t \geq 0)$ defined by
\begin{equation*}
 \bar{N}^{(z,K)}(t) =\lfloor zK\rfloor+\underset{(\alpha,\beta)\in \mathcal{E}}{\sum}\Big[{Y}_1^{(\alpha\beta)}\Big( \int_0^t{b}_{\alpha \beta}^K(\bar{N}^{(z,K)}({s})) ds\Big)-
{Y}_2^{(\alpha\beta)}\Big( \int_0^t{d}_{\alpha \beta}^K(\bar{N}^{(z,K)}({s})) ds\Big)\Big],
\end{equation*}
has the same law as $({N}^{(z,K)}(t), t \geq 0)$. 
Applying Definitions \eqref{deftildeb} and \eqref{deftilded} we get:
$$\frac{ \bar{N}^{(z,K)}(t)}{K} =\frac{\lfloor zK\rfloor}{K}+Mart^{(z,K)}(t)+\int_0^t \underset{(\alpha,\beta)\in \mathcal{E}}{\sum}
e_{\alpha \beta}\Big( \tilde{b}_{\alpha \beta}^K\Big(\frac{ \bar{N}^{(z,K)}({s})}{K}\Big) -\tilde{d}_{\alpha \beta}\Big(\frac{ \bar{N}^{(z,K)}({s})}{K}\Big) \Big)ds, $$
where we recall that $(e_{\alpha \beta}, (\alpha,\beta)\in\mathcal{E})$ is the canonical basis of $\R_+^\mathcal{E}$ and the martingale $Mart^{(z,K)}$ is defined by
$$Mart^{(z,K)} := \frac{1}{K}\underset{(\alpha,\beta)\in \mathcal{E}}{\sum}\Big[\tilde{Y}_1^{(\alpha\beta)}\Big(K \int_0^t\tilde{b}_{\alpha \beta}^K\Big(\frac{ \bar{N}^{(z,K)}({s})}{K}\Big) ds\Big)-
\tilde{Y}_2^{(\alpha\beta)}\Big(K \int_0^t\tilde{d}_{\alpha \beta}\Big(\frac{ \bar{N}^{(z,K)}({s})}{K}\Big) ds\Big)\Big]$$
and $(\tilde{Y}_i^{(\alpha\beta)}(u)={Y}_i^{(\alpha\beta)}(u)-u,u \geq 0,i \in \{1,2\}, (\alpha,\beta)\in \mathcal{E})$ are the Poisson processes centered at their expectation.
We also have by definition
$$n^{(z,K)}(t)= z+\int_0^t \underset{(\alpha,\beta)\in \mathcal{E}}{\sum}
e_{\alpha \beta}\Big( \tilde{b}_{\alpha \beta}^K(n^{(z,K)}({s})) -\tilde{d}_{\alpha \beta}(n^{(z,K)}({s})) \Big)ds. $$
Hence, for every $t\leq t_\eps(z)$,
\begin{multline*}
 \Big|\frac{\bar{N}^{(z,K)}(t)}{K} -n^{(z,K)}(t)\Big|\leq \Big|\frac{\lfloor zK \rfloor}{K}-z \Big|+  \Big|Mart^{(z,K)}(t)\Big|  \\
 +
\int_0^t \underset{(\alpha,\beta)\in \mathcal{E}}{\sum}\Big|
\Big( \tilde{b}_{\alpha \beta}^K -\tilde{d}_{\alpha \beta}\Big)\Big(\frac{\bar{N}^{(z,K)}({s})}{K}\Big) -
\Big( \tilde{b}_{\alpha \beta}^K -\tilde{d}_{\alpha \beta}\Big)\Big(n^{(z,K)}({s})\Big) \Big|ds ,
\end{multline*}
and there exists a finite constant $\mathcal{K}$ such that
\begin{equation*}
 \Big|\frac{\bar{N}^{(z,K)}(t)}{K} -n^{(z,K)}(t)\Big|\leq \frac{1}{K}+  \Big|Mart^{(z,K)}(t)\Big|
 + \mathcal{K}
\int_0^t \Big|
\frac{\bar{N}^{(z,K)}({s})}{K} - n^{(z,K)}({s}) \Big|ds.
\end{equation*}
But following Ethier and Kurtz, we get
$$ \underset{K \to \infty}{\lim}\underset{s\leq t_\eps(z)}{\sup}|Mart^{(z,K)}|=0, \quad \text{a.s.}, $$
and using Gronwall's Lemma we finally obtain
\begin{equation*}
\underset{K \to \infty}{\lim}\  \sup_{s\leq t_\eps(z)}\ \|\bar{N}^{(z,K)}(s)/K-n^{(z,K)}(s)  \|=0 \quad \text{a.s.}
\end{equation*}
As the convergence in law to a constant is equivalent to the convergence in probability to the same constant, the result follows.
\end{proof}

Once we know that the rescaled population process is close to the solution of the dynamical system \eqref{syst_dyn}, 
we can study the dynamical system.

\begin{lem}\label{lemstudysd}
Let $z$ be in $\R^{\mathcal{E}}_+$ such that $z_A>0$ and $z_a>0$. Then $n_a^{(z,K)}(t)$ and $n^{(z,K)}_{ab_1}(t)$ have a finite limit when 
$t$ goes to infinity, and 
there exists a positive constant $\eps_0$ such that for every $\eps\leq \eps_0$, 
\begin{equation*}
 \Big| \frac{n^{(z,K)}_{ab_1}(\infty)}{n^{(z,K)}_{a}(\infty)}- \frac{n^{(z,K)}_{ab_1}(t_\eps(z))}{n^{(z,K)}_{a}(t_\eps(z))} \Big|
\leq \frac{2f_a\eps^2}{\bar{n}_A|S_{aA}|}.
\end{equation*}
\end{lem}

\begin{proof}
First notice that by definition of the dynamical systems \eqref{S1} and \eqref{syst_dyn}, $n_\alpha^{(z,K)}=n_\alpha^{(z)}$ for 
$\alpha \in \mathcal{A}$ and $z \in \R^{\mathcal{E}}$.
Assumption \ref{assumption_eq} ensures that $n_a^{(z)}(t)$ goes to $\bar{n}_a$ at infinity. If we define the functions
$$p_{\alpha,b_1}^{(z,K)}= {n^{(z,K)}_{\alpha b_1}}/{{n}^{(z)}_{\alpha}}, \quad \alpha \in \mathcal{A},\quad 
\text{and} \quad g^{(z,K)}=p_{A,b_1}^{(z,K)}-p_{a,b_1}^{(z,K)},$$
we easily check that $\phi:(n_{Ab_1}^{(z,K)},n_{Ab_2}^{(z,K)},n_{ab_1}^{(z,K)},n_{ab_2}^{(z,K)})\mapsto (n_A^{(z)},n_a^{(z)}, 
g^{(z,K)},p_{a,b_1}^{(z,K)})$ defines a change of variables from $(\R_+^{*})^\mathcal{E}$ to 
$\R_+^{2*} \times ]-1,1[\times]0,1[$, and \eqref{syst_dyn} is equivalent to:
\begin{equation} \label{syst_dyn2}  
\left\{ \begin{array}{l} 
 \dot{n}^{(z)}_{\alpha} =( f_\alpha -(D_\alpha+C_{\alpha,A}{n}^{(z)}_{A}+C_{\alpha,a}{n}^{(z)}_{a})){n}^{(z)}_{\alpha},\quad  \alpha \in \mathcal{A}  \\
 \dot{g}^{(z,K)}  =-g^{(z,K)}\Big(r_K f_Af_a(n^{(z)}_A+n^{(z)}_a)/(f_An^{(z)}_A+f_an^{(z)}_a)\Big) \\
 \dot{p}_{a,b_1}^{(z,K)} =g^{(z,K)}\Big(r_K f_Af_an^{(z)}_A/(f_An^{(z)}_A+f_an^{(z)}_a)\Big) ,\end{array} \right.
 \end{equation}
with initial condition 
$$(n_A^{(z)}(0),n_a^{(z)}(0), 
g^{(z,K)}(0),p_{a,b_1}^{(z,K)}(0))=(z_A,z_a,z_{Ab_1}/z_A-z_{ab_1}/z_a ,z_{ab_1}/z_a).$$
Moreover, a direct integration yields
\begin{equation}\label{expreproba}  {p^{(z,K)}_{a,b_1}(t)}= {p^{(z,K)}_{a,b_1}(0)}-( {p_{a,b_1}^{(z,K)}(0)}-{p_{A,b_1}^{(z,K)}(0)} )F(z,r_K,t),\end{equation}
where $F$ has been defined in \eqref{defF}. According to \eqref{limF}, $F(z,r_K,t)$ has a finite limit when $t$ goes to infinity. Hence $p^{(z,K)}_{a,b_1}$ 
also admits a limit at infinity. 
Let $\eps\leq |S_{Aa}|/C_{A,a}\wedge C_{a,a}/ C_{A,a}\wedge \bar{n}_a/2$, and $t_\eps(z)$ defined in \eqref{deftepsz1}. Then for $t \geq t_\eps(z)$,
$$  \dot{n}_A^{(z)}(t)  \leq  (f_A-D_A-C_{Aa}(\bar{n}_a-\eps/2))n_A^{(z)}(t)  \leq  S_{Aa} {n}_A^{(z)}(t)/2<0.$$ 
Recalling that $r_K\leq 1$ and $|g(t)|\leq 1$ for all 
$t\geq 0$ we get:
\begin{equation} \label{diff_lim}\Big|p^{(z,K)}_{a,b_1}(\infty)-p^{(z,K)}_{a,b_1}(t_\eps(z))\Big|\leq 
\int_{t_\eps(z)}^\infty \frac{f_Af_an^{(z)}_A}{f_An^{(z)}_A+f_an^{(z)}_a}\leq \frac{f_A\eps^2}{\bar{n}_a-\eps/2}\int_0^\infty e^{S_{Aa}s/2}ds 
\leq \frac{2f_a\eps^2}{(\bar{n}_a-\eps/2)|S_{aA}|},\end{equation}
which ends the proof.\end{proof}

\subsection{$A$-population extinction}
The deterministic approximation \eqref{syst_dyn} fails when the $A$-population size becomes too small. We shall compare 
$N_A$ with birth and death processes to study the last period of the mutant invasion. We show that during this period, the number of $A$ individuals is so small
 that it has no influence on the neutral proportion in the $a$-population, which stays nearly 
constant. 
Before stating the result, we recall Definition \eqref{defText} and introduce the compact set $\Theta$:
\begin{equation}\label{defTheta}
\Theta:=\big\{ z \in \R_+^{A\times \mathcal{B}} \times \R_+^{a\times \mathcal{B}},  z_A\leq \eps^2 \quad \text{and} \quad  |z_a-\bar{n}_a |\leq \eps \big\},
\end{equation}
the constant $M''=3+(f_a+C_{a,A})/C_{a,a}$, and the stopping time:
\begin{equation}\label{SKeps}
{U}^K_\eps(z):=\inf \Big\{ t\geq 0, N^{(z,K)}_{A}(t)>   \eps K \text{ or } |N^{(z,K)}_a(t)-\bar{n}_a K|>M''\eps K \Big\}.
\end{equation}

\begin{lem}\label{third_step}
Let $z$ be in $\Theta$. Under Assumption \ref{assumption_eq},  there exist two positive finite constants $c$ and $\eps_0$ such that for $\eps\leq \eps_0$,
$$ \underset{K \to \infty}{\limsup} \ \P\Big(\underset{t \leq T^{(z,K)}_{\textnormal{ext}}}{\sup}\Big|P_{a,b_1}^{(z,K)}(t)-P_{a,b_1}^{(z,K)}(0)\Big|>\eps\Big)\leq c\eps. $$
\end{lem}

\begin{proof}
Let $z $ be in $\Theta$ and $Z^1$ be a birth and death process with individual birth rate $f_A$, individual death rate 
$D_A+(\bar{n}_a-M''\eps)C_{A,a}$, and initial state $\lceil \eps^2 K \rceil$.
Then on $[0,U_\eps^K(z)[$, $N_A$ and $Z^1$ have the same birth rate, 
and $Z^1$ has a smaller death rate than $N_A$.  Thus according to Theorem 2 in \cite{champagnat2006microscopic}, 
we can construct the processes $N$ and $Z^1$ on the same probability space such that:
\begin{equation}\label{compaZ1} N_A(t)\leq Z^1_t, \quad \forall t < U_\eps^K(z). \end{equation}
 Moreover, if we denote by $T_0^1$ the extinction time of $Z^1$, $ T_0^1:=\inf\{ t\geq 0, Z_t^1=0 \} ,$ and recall that 
\begin{equation}\label{decroissance}  f_A-D_A-(\bar{n}_a-M''\eps)C_{A,a}=S_{Aa}+M''C_{A,a}\eps<S_{Aa}/2<0, \quad \forall \eps <|S_{Aa}|/(2M''C_{Aa}), \end{equation}
we get according to (\ref{ext_times}) that for $z\leq \eps^2$ and 
$$L(\eps,K)=2 \log K/|S_{Aa}+M''\eps C_{A,a}|,$$
\begin{eqnarray*} 
 \P_{\lceil z K \rceil}\Big(T_0^1\leq L(\eps,K)  \Big) \geq \exp\Big( \lceil \eps^2 K \rceil \Big[\log (K^2-1)-\log(K^2-f_A(D_A+(\bar{n}_a-M''\eps)C_{A,a})^{-1})  \Big] \Big).
\end{eqnarray*}
Thus:
\begin{equation} \label{sim}
 \lim_{K \to \infty}\P_{\lceil z K \rceil}\Big(T_0^1< L(\eps,K)  \Big)=  1.
\end{equation}
Moreover, Equation \eqref{resSepsKz} ensures the existence of a finite $c$ such that for $\eps$ small enough,
\begin{equation} \label{LS}
 \P \Big(L(\eps,K)<U_\eps^{K}(z) \Big)\geq 1-c\eps. \end{equation}
Equations \eqref{sim} and \eqref{LS} imply
\begin{equation} \label{LS2}
 \underset{K \to \infty}{\liminf} \ \P \Big(T_0^1< L(\eps,K)<U_\eps^{K}(z) \Big)\geq 1-c\eps \end{equation}
for a finite $c$ and $\eps$ small enough. According to Coupling \eqref{compaZ1} we have the inclusion 
$$\{T_0^1< L(\eps,K)<U_\eps^{K}(z)\}\subset\{T^K_{ext}< L(\eps,K)<U_\eps^{K}(z)\}.$$
Adding \eqref{LS2} we finally get:
\begin{equation} \label{LS3}
 \underset{K \to \infty}{\liminf}\ \P(T_{ext}^K< L(\eps,K)<U_\eps^{K}(z))\geq 1-c\eps. \end{equation}
Recall the martingale decomposition of $P_{a,b_1}$ in \eqref{defM}. To bound the difference $|P_{a,b_1}(t)-P_{a,b_1}(0)|$ we bound independently 
the martingale $M_a(t)$ and the integral $|P_{a,b_1}(t)-P_{a,b_1}(0)-M_a(t)|$.
On one hand Doob's Maximal Inequality and Equation \eqref{crocheten1K1} imply:
\begin{eqnarray}\label{maj_mart} \P\Big(\underset{t \leq  L(\eps,K) \wedge U_\eps^K }{\sup}|M_a(t)|>\frac{{\varepsilon}}{2} \Big) 
& \leq & \frac{4}{\varepsilon^2}\E\Big[ \langle M_a\rangle_{ L(\eps,K)\wedge U_\eps^{K}(z)} \Big] \nonumber \\
& \leq &  \frac{4 C(a,\bar{n}_a+M''\eps)L(\eps,K)}{\eps^2K(\bar{n}_a-M''\eps)}\nonumber\\
& = & \frac{8 C(a,\bar{n}_a+M''\eps)\log K}{\eps^2K(\bar{n}_a-M''\eps)|S_{Aa}+M''\eps C_{A,a}|}. \end{eqnarray}
On the other hand the inequality $|N_{Ab_1}N_{a b_2}-N_{a b_1}N_{Ab_2}|\leq N_AN_a$ yields for $t \geq 0$
\begin{eqnarray*} \Big| \int_0^{t\wedge U_\eps^K(z)}\frac{r_Kf_Af_a (N_{Ab_1}N_{a b_2}-N_{a b_1}N_{Ab_2})}
{(N_a+1)(f_AN_A+f_aN_a)} \Big| \leq \int_0^{{t\wedge U_\eps^K(z)}} \frac{f_A N_A}{(\bar{n}_a- \eps M'')K}.
  \end{eqnarray*}
Hence decomposition \eqref{defM}, Markov's Inequality, and Equations \eqref{compaZ1}, \eqref{espviemort} and  \eqref{decroissance} yield
\begin{equation} \label{maj_int} \P\Big(\Big|(P_{a,b_1}-M_a)(t\wedge U_\eps^K(z))-P_{a,b_1}(0)\Big|>\frac{\eps}{2}\Big)\leq 
 \frac{ 2f_A\eps^2}{\eps(\bar{n}_a- \eps M'')} \int_0^{t} e^{\frac{S_{Aa}s}{2}}ds
 \leq  \frac{4f_A\eps}{(\bar{n}_a- \eps M'')|S_{Aa}|}.
\end{equation}
Taking the limit of  \eqref{maj_mart} when $K$ goes to infinity and adding \eqref{maj_int} end the proof. 
\end{proof}

\subsection{End of the proof of Theorem \ref{main_result2}}
Recall Definitions \eqref{defText} and \eqref{deftepsz1}. We have:
\begin{multline*} \Big|P_{a,b_1}^{(z,K)}(T_{\text{ext}}^{(z,K)})-p_{a,b_1}^{(z,K)}(\infty)\Big|  \leq \Big|P_{a,b_1}^{(z,K)}(T_{\text{ext}}^{(z,K)})-P_{a,b_1}^{(z,K)}(t_\eps(z))\Big|+\\ 
 \Big|P_{a,b_1}^{(z,K)}(t_\eps(z))-p_{a,b_1}^{(z,K)}(t_\eps(z))\Big|+\Big|p_{a,b_1}^{(z,K)}(t_\eps(z))-p_{a,b_1}^{(z,K)}(\infty)\Big|. \end{multline*}
To bound the two last terms we use respectively Lemmas \ref{lemapprox} and \ref{lemstudysd}. For the first term 
of the right hand side, 
 \eqref{result_champa1} ensures that with high probability, $N^{(z,K)}(t_\eps(z)) \in \Theta$ and $t_\eps(z)<T_{\text{ext}}^{(z,K)}$. Lemma \ref{third_step},
 Equation \eqref{convfixcas1} and Markov's Inequality allow us to conclude that for $\eps$ small enough
$$ \underset{K \to \infty}{\limsup}\ \P(\mathbf{1}_{\text{Fix}^{(z,K)}}|P_{a,b_1}^{(z,K)}(T_{\text{ext}}^{(z,K)})-p_{a,b_1}^{(z,K)}(\infty)|>3\eps)\leq c{\eps}, $$
for a finite $c$, which is equivalent to the convergence in probability. Adding \eqref{expreproba} completes the proof.

\section{A coupling with two birth and death processes} \label{section_couplage}

In Sections \ref{proofstrong} and \ref{proofweak} we suppose that Assumptions \ref{assumption_eq} and \ref{defrK} hold and we denote by $N^K$ the process $N^{(z^{(K)},K)}$.
As it will appear in the proof of Theorem \ref{main_result} the first period of mutant invasion, which ends at time $T_\eps^K$ when the mutant population size 
hits $\lfloor \eps K \rfloor$, is the most important for the neutral proportion dynamics. Indeed, the neutral proportion in the $a$-population has already reached its final value at time $T_\eps^K$. Let us describe a coupling of the process $N_a^K$ with two birth and death processes which will be a key argument to control 
the growing of the population $a$ during the first period. We recall Definition \eqref{TKTKeps1} and define for $\eps < S_{aA}/( 2 {C_{a,A}C_{A,a}}/{C_{A,A}}+C_{a,a} )$,
\begin{equation}\label{def_s_-s_+1}
s_-(\eps):=\frac{S_{aA}}{f_a}-\eps\frac{ 2 {C_{a,A}C_{A,a}}+C_{a,a}{C_{A,A}}}{f_a{C_{A,A}}}, \quad \text{and} \quad s_+(\eps):=\frac{S_{aA}}{f_a}
+ 2\varepsilon \frac{C_{a,A}C_{A,a}}{f_aC_{A,A}}.
\end{equation}
Definitions 
\eqref{defdaba} and \eqref{deffitinv1} ensure that for $t < T_\eps^K \wedge S_\eps^K$,
\begin{equation}\label{ineqtxmort}
f_a (1-s_+(\eps))\leq \frac{{d}_{a}^K(N^K(t))}{N_a^K(t)}= f_a-S_{aA}+\frac{C_{a,A}}{K}(N^K_A(t)-\bar{n}_AK)+\frac{C_{a,a}}{K}N^K_a(t)\leq f_a (1-s_-(\eps)),
\end{equation}
and following Theorem 2 in \cite{champagnat2006microscopic}, we can construct on the same probability space the processes 
$Z^-_\eps$, $N^K$ and $Z^+_\eps$ such that almost surely:
\begin{equation}\label{couplage11}
 Z^-_\eps(t) \leq N_a^K(t) \leq Z^+_\eps(t), \quad \text{for all } t <  T^K_\eps\wedge S^K_\eps ,
\end{equation}
where for $*\in \{-,+\}$, $Z^*_\eps$ is a birth and death process with initial state $1$, and individual birth  and death rates $f_a$ and  $f_a (1-s_*(\eps))$.


\section{Proof of Theorem \ref{main_result} in the strong recombination regime}\label{proofstrong}

In this section, we suppose that Assumptions \ref{assumption_eq}, \ref{defrK} and \ref{condstrong} hold.
We distinguish the three periods of the selective sweep: (i) rare mutants 
and resident population size near its equilibrium value, (ii) quasi-deterministic period governed by the dynamical system \eqref{S1}, and (iii) 
$A$-population extinction. First we prove that at time $T_\eps^K$ proportions of $b_1$ alleles in the populations $A$ and $a$ are close to $z_{Ab_1}/z_A$. 
Once the neutral proportions are the same in the two populations, they do not evolve anymore until the end of the sweep.

\begin{lem}\label{majespP}
 There exist two positive finite constants $c$ and $\eps_0$ such that for $\eps\leq \eps_0$:
\begin{equation*} 
  \underset{K \to \infty}{\limsup} \ \E \Big[\mathbf{1}_{ T_\eps^K\leq  S_\eps^K}\Big\{\Big|P_{A,b_1}^K( T_\eps^K)-\frac{ z_{Ab_1}}{ z_A} \Big|
+\Big|P_{A,b_1}^K( T_\eps^K)-P_{a,b_1}^K( T_\eps^K) \Big| \Big\}\Big]
\leq c\eps.
\end{equation*}
\end{lem}

\begin{proof}First we bound the difference between the neutral proportions in the two populations, $|P_{a,b_1}(t)-P_{A,b_1}(t)| $,
then we bound $|P_{A,b_1}(t)-{z_{Ab_1}}/{z_A}|$.
For sake of simplicity we introduce:
\begin{equation} \label{defG} G(t):=P_{A,b_1}(t)-P_{a,b_1}(t)=\frac{N_{ab_2}(t)N_{Ab_1}(t)-N_{ab_1}(t)N_{Ab_2}(t)}{N_{A}(t)N_{a}(t)}, \quad \forall t \geq 0  , \end{equation}
\begin{equation} \label{defY} Y(t)=G^2(t)e^{r_K f_at/2}, \quad \forall t \geq 0 .\end{equation}
Recalling \eqref{muAmua0} and applying Itô's formula with jumps we get
\begin{multline}\label{exprY}
 Y(t\wedge T_\eps^K\wedge S_\eps^K)=Y(0)+\hat{M}_{t\wedge T_\eps^K\wedge S_\eps^K}
+r_K\int_0^t\mathbf{1}_{s< T_\eps^K\wedge S_\eps^K}\Big(f_a/2-H(s) \Big)Y(s)ds\\
+\int_0^t\mathbf{1}_{s< T_\eps^K\wedge S_\eps^K}e^{r_K f_as/2}ds
\int_{\R_+}\Big[\Big(\mu_A^K(N,s,\theta)\Big)^2+\Big(\mu_a^K(N,s,\theta)\Big)^2\Big]d\theta,
\end{multline}
where $\hat{M}$ is a martingale with zero mean, and $H$ is defined by
\begin{equation} \label{defH}
H(t)=\frac{2f_a f_AN_A(t)N_a(t)}{f_AN_A(t)+f_aN_a(t)}\Big[\frac{1}{N_A(t)+1}+\frac{1}{N_a(t)+1}\Big]\geq \frac{f_a}{2},\quad t < T_\eps^K\wedge S_\eps^K,
\end{equation}
for $\eps$ small enough. In particular the first integral in \eqref{exprY} is non-positive.
Applying Lemma \ref{lemmualpha} we obtain:
\begin{eqnarray}\label{majespY}
 \E[Y(t\wedge T_\eps^K\wedge S_\eps^K)]&\leq & 1+ \frac{2C(A,2\bar{n}_A)}{r_K f_a(\bar{n}_A-2\eps C_{A,a}/C_{A,A})K}
e^{\frac{r_K f_at}{2}} \nonumber \\
&& +\int_0^t (k_0+1)C(a,2\bar{n}_A)\E\Big[ \tilde{M}_{s\wedge T_\eps^K\wedge S_\eps^K}\Big]
e^{ (\frac{r_K f_a}{2}-\frac{S_{aA}}{2(k_0+1)} )s}ds \nonumber \\
& \leq & c\Big( 1+ \frac{1}{Kr_K}e^{\frac{r_K f_at}{2}}+e^{ (\frac{r_Kf_a}{2}-\frac{S_{aA}}{2(k_0+1)} )t}  \Big),
\end{eqnarray}
where $c$ is a finite constant which can be chosen independently of $\eps$ and $K$ if $\eps$ is small enough and $K$ large enough. 
Combining the semi-martingale decomposition \eqref{defM}, the Cauchy-Schwarz Inequality, and Equations \eqref{crocheten1K1} and \eqref{majespY} we get for every $t \geq  0$,
\begin{multline*} \E \Big[\Big|P_{A,b_1}(t\wedge T_\eps^K\wedge S_\eps^K)-\frac{\lfloor z_{Ab_1}K\rfloor}{\lfloor z_A K\rfloor}\Big|\Big]
 \\ \leq  \E \Big[|M_A(t\wedge T_\eps^K\wedge S_\eps^K)|\Big]+
\frac{r_K f_a\eps }{\bar{n}_A-2\eps C_{A,a}/C_{A,A}} \int_0^t \E \Big[\mathbf{1}_{s< T_\eps^K\wedge S_\eps^K}|G(s)|\Big]ds\\
\leq  \E^{1/2} \Big[\langle M_A \rangle_{t\wedge T_\eps^K\wedge S_\eps^K}\Big]
+cr_K\eps  \int_0^t \E^{1/2}\Big[Y(s\wedge T_\eps^K\wedge S_\eps^K)\Big]e^{-r_Kf_as/4}ds
\\
\leq  c\Big(\sqrt{t/K}+\eps r_K \int_0^t \Big(e^{-r_K f_as/2}+\frac{1}{Kr_K}+e^{-S_{aA}s/2(k_0+1)}\Big)^{1/2}ds\Big),
\end{multline*}
where $c$ is finite. A simple integration then yields the existence of a finite $c$ such that:
\begin{eqnarray} \label{majdiffprop}\E \Big[\Big|P_{A,b_1}(t\wedge T_\eps^K\wedge S_\eps^K)-
\frac{\lfloor z_{Ab_1}K\rfloor}{\lfloor z_A K\rfloor}\Big|\Big]
  \leq c\Big(\sqrt{\frac{t}{K}}+\eps\Big(1+  \frac{t}{\sqrt{K}}\Big)\Big).
\end{eqnarray}
Let us introduce the sequences of times
$$t_K^{(-)}=(1-c_1\eps)\frac{\log K}{S_{aA}}, \quad \text{and} \quad t_K^{(+)}=(1+c_1\eps)\frac{\log K}{S_{aA}},$$
where $c_1$ is a finite constant. Then according to Coupling \eqref{couplage11} and limit
\eqref{equi_hitting},
\begin{equation} \label{2slog} \underset{K \to \infty}{\lim} \P(T_\eps^K < t_K^{(-)}|T_\eps^K \leq S_\eps^K)=
 \underset{K \to \infty}{\lim} \P(T_\eps^K > t_K^{(+)}|T_\eps^K \leq S_\eps^K)=0.\end{equation}
Hence applying \eqref{majdiffprop} at time $t_K^{(+)}$
 and using \eqref{res_champ} and \eqref{2slog}, we bound the first term in 
the expectation.
To bound the second term in the expectation, we introduce the notation
$$ A(\eps,K):=\E \Big[\mathbf{1}_{t_K^{(-)} \leq  T_\eps^K\leq  S_\eps^K\wedge t_K^{(+)}}
\Big|P_{A,b_1}^K( T_\eps^K\wedge S_\eps^K)-P_{a,b_1}^K( T_\eps^K\wedge S_\eps^K) \Big|\Big]. $$
From \eqref{2slog} we obtain
\begin{equation} \label{reducbound2}  \underset{K \to \infty}{\limsup} \ \E \Big[\mathbf{1}_{ T_\eps^K\leq  S_\eps^K}
\Big|P_{A,b_1}^K( T_\eps^K)-P_{a,b_1}^K( T_\eps^K) \Big|\Big]=\underset{K \to \infty}{\limsup} \ A(\eps,K), \end{equation}
and by using \eqref{defY}, the Cauchy-Schwarz Inequality, and \eqref{majespY} we get
\begin{eqnarray*}
 A(\eps,K)&\leq  &\E \Big[\sqrt{Y(t_K^{(+)}\wedge T_\eps^K\wedge S_\eps^K)} \Big] e^{-\frac{r_Kf_a}{4}t_K^{(-)}}\\
&\leq  &\E^{1/2} \Big[{Y(t_K^{(+)}\wedge T_\eps^K\wedge S_\eps^K)} \Big] e^{-\frac{r_Kf_a}{4}t_K^{(-)}}\\
&\leq & c\Big( 1+ \frac{1}{Kr_K}e^{r_Kf_at_K^{(+)}/2}+e^{ (\frac{r_K f_a}{2}-\frac{S_{aA}}{2(k_0+1)} )t_K^{(+)}}  \Big)^{1/2}
e^{-\frac{r_Kf_a}{4}t_K^{(-)}}\\
&\leq & c\Big( e^{-\frac{r_Kf_a}{4}t_K^{(-)}}+ \frac{1}{\sqrt{Kr_K}}e^{\frac{c_1\eps r_Kf_a\log K}{2S_{aA}}}+
e^{ (\frac{c_1\eps r_Kf_a\log K}{2S_{aA}}-\frac{S_{aA}}{4(k_0+1)} t_K^{(+)})}  \Big),
\end{eqnarray*}
where the value of the constant $c$ can change from line to line. Assumption \ref{condstrong} then yields
$$ \underset{K \to \infty}{\limsup} \ A(\eps,K)=0, $$
and we end the proof of the second bound by applying \eqref{reducbound2}. 
\end{proof}

The following Lemma states that during the second period, the neutral proportion stays constant in the $a$-population.

\begin{lem}\label{majdet}
There exist two positive finite constants $c$ and $\eps_0$ such that for $\eps\leq \eps_0$:
\begin{equation*} 
 \underset{K \to \infty}{\limsup}
 \ \E\Big[\mathbf{1}_{ T_\eps^K\leq S_\eps^K}\Big|P^K_{a,b_1}(T_\eps^K+t_\eps(\textstyle{\frac{N^K(T_\eps^K)}{K}}))-\frac{z_{Ab_1}}{z_A}\Big|\Big]\leq c\eps.
\end{equation*} 
\end{lem}

\begin{proof}
Let us introduce, for $z \in \R_+^\mathcal{E}$ and $\eps>0$ the set $\Gamma$ and the time $t_\eps$ defined as follows:
\begin{equation}\label{tetaCfini1}\Gamma:=\Big\{ z \in \R_+^\mathcal{E}, \Big|z_A- \bar{n}_A\Big|\leq 2\eps \frac{C_{A,a}}{C_{A,A}}, \Big|z_a-\eps\Big|\leq \frac{\eps}{2}  \Big\}, \quad  t_\eps:=\sup \{ t_\eps(z),z \in \Gamma\},
\end{equation}
where $t_\eps(z)$ has been defined in \eqref{deftepsz1}. According to Assumption \ref{assumption_eq}, $t_\eps<\infty$, and  
$$ I(\Gamma,\eps):=\underset{z \in \Gamma}{\inf} \ \underset{t \leq t_\eps}{\inf}\ \{n_A^{(z)}(t),n_a^{(z)}(t)\}>0,$$
and we can introduce the stopping time
\begin{equation}\label{defL}
 L_\eps^K(z)=\inf\big\{ t\geq 0, (N_A^{(z,K)}(t),N_a^{(z,K)}(t)) \notin [ I(\Gamma,\eps)K/2,(\bar{n}_A+\bar{n}_a)K]^2 \big\}.
\end{equation}
Finally, we denote by $(\mathcal{F}_t^K, t\geq 0)$ the canonical filtration of $N^K$.
Notice that on the event $\{ T_\eps^K\leq S_\eps^K \}$, $N(T_\eps^K)/K \in \Gamma$, thus $t_\eps(\textstyle{{N(T_\eps^K)}/{K}})\leq t_\eps$. 
The semi-martingale decomposition \eqref{defM} and the definition of $G$ in \eqref{defG} then
twice the Strong Markov property and the Cauchy-Schwarz Inequality yield:
\begin{multline}\label{majdebut}
\E\Big[ \mathbf{1}_{ T_\eps^K\leq  S_\eps^K}\Big|P_{a,b_1}\Big(T_\eps^K+t_\eps(\textstyle{\frac{N(T_\eps^K)}{K}})\wedge
 L_\eps^K(\textstyle{\frac{N(T_\eps^K)}{K}})\Big)-P_{a,b_1}(T_\eps^K)\Big|\Big]
\\ \leq  \E\Big[ \mathbf{1}_{ T_\eps^K\leq S_\eps^K}\E\Big[ \Big|M_a\Big(T_\eps^K+t_\eps(\textstyle{\frac{N(T_\eps^K)}{K}})\wedge 
L_\eps^K(\textstyle{\frac{N(T_\eps^K)}{K}})-M_a(T_\eps^K)\Big| +f_a \displaystyle\int_{T_\eps^K}^{T_\eps^K +t_\eps \wedge 
L_\eps^K(\textstyle{\frac{N(T_\eps^K)}{K}})}  |G|\Big|\mathcal{F}_{T_\eps^K}\Big]\Big]
\\ \leq  \E\Big[ \mathbf{1}_{ T_\eps^K\leq S_\eps^K}\Big\{\E^{1/2}\Big[ \langle M_a\rangle_{T_\eps^K+t_\eps\wedge 
L_\eps^K(\textstyle{\frac{N(T_\eps^K)}{K}})}-\langle M_a\rangle_{T_\eps^K} \Big|\mathcal{F}_{T_\eps^K}\Big]+f_a \sqrt{t_\eps} \E^{1/2}\Big[\int_{T_\eps^K}^{T_\eps^K +t_\eps \wedge 
L_\eps^K(\textstyle{\frac{N(T_\eps^K)}{K}})}  G^2\Big|\mathcal{F}_{T_\eps^K}\Big]\Big\}\Big].
\end{multline}
To bound the first term of the right hand side we use the Strong Markov Property, Equation \eqref{crocheten1K1} and the definition of $L_\eps^K$ in \eqref{defL}. 
We get
\begin{equation}\label{majt1}  \E\Big[ \mathbf{1}_{ T_\eps^K\leq S_\eps^K}\E^{1/2}\Big[ \langle M_a\rangle_{T_\eps^K+t_\eps\wedge 
L_\eps^K(\textstyle{\frac{N(T_\eps^K)}{K}})}-\langle M_a\rangle_{T_\eps^K} \Big|\mathcal{F}_{T_\eps^K}\Big]\Big]\leq \sqrt{ \frac{2t_\eps C(a,\bar{n}_A+\bar{n}_a)}{I(\Gamma,\eps)K}}. \end{equation}
Let us now focus on the second term. Itô's formula with jumps yields for every $t\geq 0$,
$$ \E\Big[G^2\Big(t \wedge 
L_\eps^K(\textstyle{\frac{N(0)}{K}})\Big)\Big]\leq \E[G^2(0)]+\E\Big[\langle M_A\rangle_{t \wedge 
L_\eps^K(\textstyle{\frac{N(0)}{K}})}\Big]+\E\Big[\langle M_a\rangle_{t \wedge 
L_\eps^K(\textstyle{\frac{N(0)}{K}})}\Big], $$
and adding the Strong Markov Property we get
\begin{multline*}
\mathbf{1}_{ T_\eps^K\leq  S_\eps^K} \E^{1/2}\Big[\int_{T_\eps^K}^{T_\eps^K +t_\eps \wedge 
L_\eps^K(\textstyle{\frac{N(T_\eps^K)}{K}})}  G^2\Big|\mathcal{F}_{T_\eps^K}\Big] \\
\leq \underset{z \in \Gamma}{\sup}\ \E^{1/2}\Big[\int_{0}^{t_\eps} \Big( G^2\Big(s \wedge 
L_\eps^K(\textstyle{\frac{N(0)}{K}})\Big)-G^2(0)\Big)ds\Big|N(0)=\lfloor z K \rfloor\Big] + \mathbf{1}_{ T_\eps^K\leq  S_\eps^K}\sqrt{t_\eps}|G(T_\eps^K)| \\
\leq \underset{z \in \Gamma}{\sup}\Big[\int_{0}^{t_\eps}  \E\Big[\langle M_A\rangle_{s \wedge 
L_\eps^K(\textstyle{\frac{N(0)}{K}})}+\langle M_a\rangle_{s \wedge 
L_\eps^K(\textstyle{\frac{N(0)}{K}})}\Big]\Big|N(0)=\lfloor z K \rfloor\Big]ds \Big]^{1/2}+ \mathbf{1}_{ T_\eps^K\leq  S_\eps^K}\sqrt{t_\eps}|G(T_\eps^K)|.
\end{multline*}
Using again Equation \eqref{crocheten1K1} and the definition of $L_\eps^K$ in \eqref{defL}, and adding Lemma \ref{majespP} finally lead to
\begin{equation}\label{majt2} \E\Big[ \mathbf{1}_{ T_\eps^K\leq S_\eps^K} \E^{1/2}\Big[\int_{T_\eps^K}^{T_\eps^K +t_\eps \wedge 
L_\eps^K(\textstyle{\frac{N(T_\eps^K)}{K}})}  G^2\Big|\mathcal{F}_{T_\eps^K}\Big]\Big]
\leq c\Big(\frac{1}{\sqrt{K}}+\eps\Big),\end{equation}
for $\eps$ small enough and $K$ large enough, where $c$ is a finite constant.
Moreover \eqref{result_champa1} ensures that
$$ \P\Big({ T_\eps^K\leq  S_\eps^K}, L_\eps^K(\textstyle{\frac{N(T_\eps^K)}{K}})\leq t_\eps(\textstyle{\frac{N(T_\eps^K)}{K}})\Big)
\leq \P\Big(\frac{N(T_\eps^K)}{K} \in \Theta,L_\eps^K(\textstyle{\frac{N(T_\eps^K)}{K}})\leq t_\eps(\textstyle{\frac{N(T_\eps^K)}{K}})\Big)\underset{K \to \infty}{\to}0,
 $$
where $\Theta$ has been defined in \eqref{defTheta}.
Adding Equations \eqref{majdebut}, \eqref{majt1}, \eqref{majt2} and Lemma \ref{majespP}, we finally end the proof of Lemma \ref{majdet}.
\end{proof}

\begin{proof}[Proof of Theorem \ref{main_result} in the strong recombination regime]
 Let us focus on the A-population extinction period. We have thanks to the Strong Markov Property:
\begin{multline}
  \P \Big( \mathbf{1}_{N(T_\eps^K+t_\eps(N(T_\eps^K)/K))\in \Theta}\Big|P_{a,b_1}(T_{\text{ext}}^K)-
P_{a,b_1}(T_\eps^K+t_\eps(\textstyle{\frac{N(T_\eps^K)}{K}}))\Big|>\sqrt{\eps}\Big)\\
\leq \underset{z \in \Theta}{\sup}\ \P \Big(|P_{a,b_1}(T_{\text{ext}}^K)-P_{a,b_1}(0)|>\sqrt{\eps}\Big| N(0) =\lfloor zK \rfloor\Big).
\end{multline}
But Equation 
\eqref{result_champa1} yields $\P(N(T_\eps^K+t_\eps(N(T_\eps^K)/K))/K \in \Theta|N(T_\eps^K)/K\in \Gamma)\to_{K\to \infty}1 $, 
and $ \{ T_\eps^K\leq S_\eps^K\} \subset  \{N(T_\eps^K)/K\in \Gamma\} $. Adding Equation \eqref{gdeprobabfix} and Lemmas \ref{third_step} and \ref{majdet}, 
the triangle inequality allows us to conclude 
that for $\eps$ small enough
 \begin{equation*}
 \limsup_{K \to \infty} \hspace{.1cm}\P\Big(\Big|P_{a,b_1}^K(T^K_{\text{ext}})-\frac{z_{Ab_1}}{z_A}\Big|> \sqrt{\eps}\Big| \text{Fix}^K\Big)
\leq c\eps.
 \end{equation*} 
As $\P(\text{Fix}^K)\to_{K \to \infty} S_{aA}/f_a>0$, it is equivalent to the claim of Theorem \ref{main_result} in the strong regime.
\end{proof}

\section{Proof of Theorem \ref{main_result} in the weak recombination regime}\label{proofweak}

\subsection{Coupling with a four dimensional population process and structure of the proof}
In this section we suppose that Assumptions \ref{assumption_eq}, \ref{defrK} and \ref{condweak} hold. 
To lighten the proofs of Sections \ref{coal_reco} to \ref{bedescended} we introduce a coupling of the population process $N$ with a process $\tilde{N}=(\tilde{N}_{\alpha\beta},(\alpha,\beta) \in \mathcal{E})$ defined as follows for every $t \geq 0$:
\begin{eqnarray}
 \tilde{N}(t) &=& \mathbf{1}_{t < S_\eps^K}N(t)+ \mathbf{1}_{t \geq  S_\eps^K} \Big( e_{Ab_1}N_{Ab_1}((S_\eps^K)^-)+ e_{Ab_2}N_{Ab_2}((S_\eps^K)^-) \\
&&+ 
\int_0^t\int_{R_+} \Big\{ e_{a b_1}\mathbf{1}_{\theta\leq b^K_{a b_1}(\tilde{N}({s^-}))}+e_{a b_2}\mathbf{1}_{b^K_{a b_1}(\tilde{N}({s^-}))<\theta\leq f_a \tilde{N}_a({s^-})}\nonumber \\
&&\hspace{1cm}\hspace{1cm}-e_{a b_1}\mathbf{1}_{0<\theta-f_a \tilde{N}_a(s^-)\leq d_{a b_1}^K(\tilde{N}({s^-}))}\nonumber \\
&&\hspace{1cm}\hspace{1cm}-e_{a b_2}\mathbf{1}_{d^K_{a b_1}(\tilde{N}({s^-}))<\theta-f_a \tilde{N}_a(s^-)\leq d_a^K(\tilde{N}({s^-}))}\Big\} Q(ds,d\theta)
\Big),
\end{eqnarray}
where the Poisson random measure $Q$ has been introduced in \eqref{defN}.
From \eqref{tildeTbiggerT} we know that 
\begin{equation}\label{couplage2}\underset{K \to \infty}{\limsup} \ \P(\{\exists t \leq T_{\eps}^K, N(t) \neq \tilde{N}(t)\}, T_{\eps}^K<\infty)\leq c \eps.\end{equation}
Hence we will study the process $\tilde{N}$ and deduce from this study properties of the dynamics of the process $N$ during the first phase.
Moreover, as
we want to prove convergences on the fixation event $\text{Fix}^K$, defined in \eqref{defText}, inequalities 
 \eqref{gdeprobabfix} and \eqref{couplage2} allow us, to study the dynamics of $\tilde{N}$ during the first phase, to restrict our attention 
 to the conditional probability measure:
 \begin{equation}\label{defhatP}
  \hat{\P}(.)=\P(.|\tilde{T}^K_{\eps} <\infty),
 \end{equation}
where $\tilde{T}^K_{\eps} $ is the hitting time of $\lfloor \eps K \rfloor$ by the process $\tilde{N}_a$:
\begin{equation} \label{TKTKeps2} \tilde{T}^K_\eps := \inf \Big\{ t \geq 0, \tilde{N}^K_a(t)= \lfloor \eps K \rfloor \Big\}. \end{equation}
Expectations and variances associated 
 with this probability measure are denoted by $\hat{\E}$ and $\hat{\var}$ respectively.

Let us notice that, as by definition $\tilde{N}_A(t)\in I_\eps^K $ for all $t \geq 0$,
 Coupling \eqref{couplage11} with birth and death processes $Z^-_\eps$ and $Z^+_\eps$ holds up to time $\tilde{T}_\eps^K$ for the process $\tilde{N}$:
\begin{equation}\label{couplage12}
 Z^-_\eps(t) \leq \tilde{N}_a^K(t) \leq Z^+_\eps(t), \quad \text{for all } t < \tilde{T}^K_\eps.
\end{equation}\\

The sketch of the proof is the following. We first focus on the neutral proportion in the $a$ population 
at time $\tilde{T}_\eps^K$. The idea is to consider the neutral alleles of the $a$ individuals at time $\tilde{T}_\eps^K$ and follow their ancestral lines back until 
the beginning of the sweep, to know whether they are descended from the first mutant or not.
Two kinds of events can happen to a neutral lineage: coalescences and m-recombinations 
(see Section \ref{coal_reco}); 
we show that we can neglect the coalescences and the occurrence of several m-recombinations for a lineage during the first period. Therefore, 
our approximation of the genealogy is the following: two neutral lineages are independent, and each of them undergoes one recombination with an 
$A$-individual during the first period with probability $\rho_K$. If it has undergone a recombination with an $A$-individual, 
it can be an allele $b_1$ or $b_2$. Otherwise it is descended from the first mutant and is an allele $b_1$.
To get this approximation we follow the approach presented by Schweinsberg 
and Durrett in \cite{schweinsberg2005random}. 
In this paper, the authors described the population dynamics by a variation of Moran model with two loci and recombinations. In their model, 
the population size was 
constant and each individual has a constant selective advantage, $0$ or $s$. In our model the size is varying and each 
individual's ability to survive and give birth depends on the population state. After the study of the first period we check that the second and third periods have little influence on the neutral 
proportion in the $a$-population.

\subsection{Coalescence and m-recombination times} \label{coal_reco}

Let us introduce the jump times of the stopped Markov process $(\tilde{N}^K(t),t\leq \tilde{T}_\eps^K)$, 
$0=:\tau_0^K<\tau_1^K<...< \tau_{J^K}^K:=\tilde{T}_\eps^K$, where $J^K$ denotes the jump number of $\tilde{N}^K$ between $0$ and $\tilde{T}_\eps^K$, and the time of the 
$m$-th jump is:
\begin{equation*}\label{deftpssauts}
 \tau_m^K= \inf \Big\{ t> \tau_{m-1}^K, \tilde{N}^K(t)\neq \tilde{N}^K(\tau_{m-1}^K) \Big\},\hspace{.1cm} 1\leq  m \leq J^K.
\end{equation*}

Let us sample two individuals with the $a$ allele uniformly at random at time $\tilde{T}_\eps^K$ and denote by $\beta_p$ and $\beta_q$ their neutral alleles.
 We want to follow their genealogy backward in time and know at each time between $0$ and $\tilde{T}_\eps^K$ the types ($A$ or $a$) of the individuals
carrying $\beta_p$ and $\beta_q$.

We say that $\beta_p$ and $\beta_q$ coalesce at time $\tau_m^K$ if they are carried by two different individuals at time $\tau_{m}^K$ and by the same 
individual at time $\tau_{m-1}^K$. In other words the individual carrying the allele $\beta_p$ (or $\beta_q$) at time $\tau_{m}^K$ is a newborn and 
has inherited his neutral allele from the individual carrying allele $\beta_q$ (or $\beta_p$) at time $\tau_{m-1}^K$.
The jump number at the coalescence time is denoted by
\begin{equation*}\label{tpscoal}
 TC^K(\beta_p,\beta_q):=\left\{\begin{array}{ll}\sup\{ m\leq J^K, \beta_p \text{ and } \beta_q \text{ coalesce at time } \tau_m^K \},& \text{ if $\beta_p$ and $\beta_q$ coalesce }\\
-\infty,& \text{ otherwise}.  \end{array}\right.
\end{equation*}

We say that $\beta_p$ m-recombines at time $\tau_m^K$ if the individual carrying the allele $\beta_p$ at time $\tau^K_m$ is a newborn,
 carries the allele $\alpha \in \mathcal{A}$, and has inherited his allele $\beta_p$ from an individual 
carrying allele $\bar{\alpha}$. In other words, a m-recombination is a recombination which modifies the selected allele connected to the neutral 
allele.
The jump numbers of the first and second (backward in time) m-recombinations are denoted by:
\begin{equation*}\label{tpsreco}
TR^K_1(\beta_p):=\left\{\begin{array}{ll}
                  \sup\{ m\leq J^K, \beta_p \text{ m-recombines at time } \tau_m^K\} ,& \text{if there is at least one m-recombination}\\
-\infty, &\text{otherwise},
  \end{array}\right.
\end{equation*}
\begin{equation*}\label{tpsreco}
TR^K_2(\beta_p):=\left\{\begin{array}{ll}
                  \sup\{ m<TR^K_1(\beta_p), \beta_p \text{ m-recombines at time } \tau_m^K\} ,& \text{if there are at least two}\\
& \text{m-recombinations}\\
-\infty,& \text{otherwise}.
  \end{array}\right.
\end{equation*}

Let us now focus on the probability for a coalescence to occur conditionally on the state of the process $(\tilde{N}_A,\tilde{N}_a)$ at two successive 
jump times. We denote by 
$p_{\alpha_1\alpha_2}^{c_K}(n)$ the probability that the genealogies of two uniformly sampled neutral alleles associated respectively with alleles $\alpha_1$ and 
$\alpha_2\in \mathcal{A}$ 
at time $\tau_m^K$ coalesce at this time conditionally on $(\tilde{N}_A^K(\tau_{m-1}^K),\tilde{N}_a^K(\tau_{m-1}^K))=n\in \N^2$ and on the birth of an individual
carrying allele $\alpha_1 \in \mathcal{A}$ at time $\tau_{m}^K$. Then we have the following result:


\begin{lem}\label{lempcoal}
For every $n=(n_A,n_a) \in \N^2$ and $\alpha \in \mathcal{A}$, we have:
\begin{equation}\label{expr_reco1}  p_{\alpha \alpha}^{c_K}(n)=\frac{2}{n_\alpha(n_\alpha+1)}\Big( 1-\frac{r_K f_{\bar{\alpha}} n_{\bar{\alpha}}}
{f_An_A+f_an_a} \Big)\quad \text{and} \quad  p_{\alpha \bar{\alpha}}^{c_K}(n) = \frac{r_Kf_{\bar{\alpha}}}{(n_\alpha+1)(f_An_A+f_an_a)}. 
 \end{equation}
 \end{lem}

\begin{proof}
We only state the expression of $p_{\alpha \alpha}^{c_K}(n)$, as the calculations are similar for $p_{\alpha \bar{\alpha}}^{c_K}(n)$. If there is a m-recombination, we cannot have the coalescence of two neutral alleles associated with allele $\alpha$ at time 
$\tau_m^K$. With probability $1-r_K f_{\bar{\alpha}}n_{\bar{\alpha}}/(f_An_A+f_an_a)$ there is no 
m-recombination and the parent giving its neutral allele carries the allele ${\alpha}$. When there is no m-recombination, two individuals
 among those who carry allele
 $\alpha$ also carry a neutral allele which was in the same individual at time $\tau_{m-1}^K$. We have a probability ${2}/{n_\alpha(n_\alpha+1)}$ 
to pick this couple of individuals 
among the $(n_\alpha+1)$ $\alpha$-individuals.\end{proof}

\begin{rem}
A m-recombination for a neutral allele associated with an $\alpha$ allele is a coalescence with an $\bar{\alpha}$
 individual. Thus if we denote by $p_\alpha^{r_K}(n)$  the probability that an 
$\alpha$-individual, chosen uniformly at time $\tau_{m}^K$, is the newborn and underwent a m-recombination at his birth, conditionally on 
$(\tilde{N}_A^K(\tau_{m-1}^K),\tilde{N}_a^K(\tau_{m-1}^K))=n \in \N^2$ and on the birth of an individual $\alpha$ at time $\tau_{m}^K$ we get
\begin{equation}\label{expr_reco2}
 p_\alpha^{r_K}(n)=n_{\bar{\alpha}}p_{\alpha \bar{\alpha}}^{c_K}(n) = \frac{n_{\bar{\alpha}}r_Kf_{\bar{\alpha}}}{(n_\alpha+1)(f_An_A+f_an_a)}.
\end{equation}
Moreover, if we recall the definition of $I_\eps^K$ in \eqref{compact1}, we notice that there exists a finite constant $c$ such that for 
$k < \lfloor \eps K \rfloor$,
\begin{equation} \label{rqpr1}
 (1-c\eps)\frac{r_K}{k+1} \leq \inf_{n_A \in I_\eps^K}\hspace{.2cm}  p_a^{r_K}(n_A,k) \leq \sup_{n_A \in I_\eps^K}\hspace{.2cm}  p_a^{r_K}(n_A,k)\leq \frac{r_K}{k+1}.
\end{equation}
\end{rem}

\subsection{Jumps of mutant population during the first period} \label{jumps}

We want to count  the number of coalescences and m-recombinations in the lineages of the two uniformly sampled neutral alleles 
$\beta_p$ and $\beta_q$. By definition, 
these events can only occur at a birth time. Thus we need to study the upcrossing number of the process $\tilde{N}_a^K$ before 
 $\tilde{T}_\varepsilon^K$ 
(Lemma \ref{uphold}). It allows us to prove that the probability 
that a lineage is affected by two m-recombinations or that two lineages coalesce, and then (backward in time) are
 affected by a m-recombination is negligible (Lemma \ref{deuxreco}). 
 Then we obtain an approximation of the probability that a lineage is affected by a
 m-recombination (Lemma \ref{lemvareta}), and finally we check that two lineages are approximately independent (Equation \eqref{dpdce}).
The last step consists in controlling the neutral proportion in the population $A$ (Lemma \ref{lemmajpro}). Indeed it will give us the probability 
that a neutral allele which has undergone a m-recombination is a $b_1$ or a $b_2$.\\

Let us denote by  $\zeta_k^K$ the jump number of last visit to $k$ before the hitting of $\lfloor \eps K \rfloor$, 
\begin{equation}\label{zeta} \zeta_k^K: = \sup \{m \leq J^K, \tilde{N}_a^K(\tau_m^K)=k \}, \quad 1\leq k \leq \lfloor  \varepsilon K \rfloor. \end{equation}
This allows us to introduce for $0< j\leq k<\lfloor  \varepsilon K \rfloor$ the number of upcrossings from $k$ to $k+1$  for the 
process $\tilde{N}_a^K$ before and after the last visit to $j$:
\begin{equation} \label{U1} U_{j,k}^{(K,1)}:=\# \{m  \in \{0,...,\zeta_j^K-1\}, (\tilde{N}_a^K(\tau_m^K),\tilde{N}_a^K(\tau_{m+1}^K))=(k,k+1) \},\end{equation}
\begin{equation} \label{U2} U_{j,k}^{(K,2)}:=\# \{m \in \{\zeta_j^K,...,J^K-1\}, (\tilde{N}_a^K(\tau_m^K),\tilde{N}_a^K(\tau_{m+1}^K))=(k,k+1) \}.\end{equation}
We also introduce the number of jumps of the $A$-population size when there are $k$ $a$-individuals and 
the total number of upcrossings from $k$ to $k+1$ before $\tilde{T}_\eps^K$:
\begin{equation} \label{H1} H_{k}^K:=\# \{m < J^K, \tilde{N}_a^K(\tau_m^K)=\tilde{N}_a^K(\tau_{m+1}^K)=k \},\end{equation}
\begin{equation} \label{Uk} U_{k}^K:=U_{j,k}^{(K,1)}+U_{j,k}^{(K,2)}=\# \{m < J^K, (\tilde{N}_a^K(\tau_m^K),\tilde{N}_a^K(\tau_{m+1}^K))=(k,k+1) \}.\end{equation}
The next Lemma states moment properties of these jump numbers. 
Recall Definition (\ref{def_s_-s_+1}). Then if we define 
\begin{equation}\label{deflambda}  \lambda_\eps:=\frac{(1-s_-(\eps))^3}{(1-s_+(\eps))^{2}}, \end{equation}
which belongs to $(0,1)$ for $\eps$ small enough, we have

\begin{lem}\label{uphold}
There exist two positive and finite constants $\eps_0$ and $c$ such that for $\eps\leq \eps_0$, $K$ large enough 
and $1\leq j \leq k< \lfloor  \varepsilon K \rfloor$,
\begin{equation} \label{E''U}  \hat{\E}[H_{j}^{K}]\leq \frac{12f_A\bar{n}_AK}{s^4_-(\eps)f_aj} ,
 \quad  \hat{\E} [ (U_{j,k}^{(K,1)})^2 ]\leq  \frac{4 \lambda_\eps^{k-j}}{ s^7_-(\eps)(1-s_+(\eps))},
\end{equation}
\begin{equation} \label{majcov}  \hat{\E}[(U_{j}^{K})^2]\leq \frac{2}{s^2_-(\eps)} ,\quad \Big|\hat{\cov}(U_{j,k}^{(K,2)},U_{j}^{K})\Big|\leq c(\eps+(1-s_-( \eps))^{k-j}),\end{equation}
and
\begin{equation}\label{espnoreco}
 r_K\Big|  \sum_{k=1}^{\lfloor \eps K \rfloor -1} \frac{\hat{\E}[U_k^K]}{k+1}-\frac{f_a\log K}{S_{aA}} \Big|\leq c\eps.
\end{equation}

\end{lem}
This Lemma is widely used in Sections \ref{recocoal} and \ref{bedescended}. Indeed, we shall decompose on the possible states of the population when a 
birth occurs, and apply Equations \eqref{expr_reco1} and \eqref{expr_reco2} 
to express the probability of coalescences and m-recombinations at each birth event. The proof of Lemma \ref{uphold} is quite technical and is 
postponed to Appendix \ref{prooflemma}.

\subsection{Negligible events}\label{recocoal}

The next Lemma bounds the probability that two m-recombinations occur in a neutral lineage and the probability that a couple 
of neutral lineages coalesce and then m-recombine when we consider the genealogy backward in time.
\begin{lem}\label{deuxreco}
There exist two positive finite constants $c$ and $\eps_0$ such that for $K \in \N$ and $\eps\leq \eps_0$,
  \begin{equation*}\label{prop24}
 \hat{\P}\Big(TR_2^K(\beta_p)\neq -\infty\Big)\leq \frac{c}{\log K}, \quad \text{and} \quad \hat{\P}\Big(0 \leq TR^K_1(\beta_p) 
\leq TC^K(\beta_p,\beta_q)\Big)\leq \frac{c}{\log K}. \end{equation*}
\end{lem}

\begin{proof}
By definition, the neutral allele $\beta_p$ is associated with an allele $a$ at time $\tilde{T}_\eps^K$. If there are at least two m-recombinations it implies 
that there exists a time between $0$ and $\tilde{T}_\eps^K$ at which $\beta_p$ has undergone 
a m-recombination when it was associated with an allele $A$. 
We shall work conditionally on the stopped process $((\tilde{N}_A(\tau_m^K),\tilde{N}_a(\tau_m^K)), m \leq J^K)$ and decompose according 
to the $a$-population size when this m-recombination occurs. We get the inclusion:
$$ \{ TR_2^K(\beta_p)\neq -\infty \} \subset \underset{k=1}{\overset{\lfloor \eps K\rfloor-1}{\bigcup}}\underset{m=1}{\overset{J^K}{\bigcup}} 
\Big\{ TR_2^K(\beta_p)=m, \tilde{N}_a(\tau^K_{m-1})=\tilde{N}_a(\tau^K_{m})=k \Big\}.  $$
We recall the definition of 
$I_\eps^K$ in \eqref{compact1}. Thanks to Equations (\ref{expr_reco2}) and 
\eqref{E''U}, we get:
\begin{eqnarray*}
 \hat{\P}(TR_2^K(\beta_p)\neq -\infty)\leq  \underset{k=1}{\overset{\lfloor \varepsilon K \rfloor-1}{\sum}}
\sup_{n_A\in I_\eps^K} p_A^{r_K}(n_A,k)\hat{\E}[H_{k}^{K}]\leq  \frac{12r_K\bar{n}_A\eps}{s^4_-(\eps)(\bar{n}_A-2\eps C_{A,a}/C_{A,A})^2} .
\end{eqnarray*}
Assumption \ref{condweak} on weak recombination completes the proof of the first inequality in Lemma \ref{deuxreco}.
The proof of the second one is divided in two steps, presented after introducing, for $(\alpha,\alpha') \in \mathcal{A}^2, m \leq J^K$ the notations
$$ (\alpha \beta_p)_m:=\{ \text{the neutral allele }\beta_p \text{ is associated with the allele } \alpha \text{ at time } \tau_m^K \}, $$
$$ (\alpha \beta_p,\alpha' \beta_q)_m:=(\alpha \beta_p)_m \cap (\alpha' \beta_q)_m .$$

\noindent \textit{First step:} We show that the probability that $\beta_p$ is associated with an allele $A$ at the coalescence time is negligible. We first recall the inclusion,
$$ \{ TC^K(\beta_p,\beta_q)\neq -\infty, (A\beta_p)_{TC^K(\beta_p,\beta_q)} \} \subset \underset{k=1}{\overset{\lfloor \eps K\rfloor-1}{\bigcup}}\underset{m=1}{\overset{J^K}{\bigcup}} 
\Big\{ TC^K(\beta_p,\beta_q)=m, \tilde{N}_a(\tau^K_{m-1})=k, (A\beta_p)_m \Big\} , $$
and decompose on the possible selected alleles associated with $\beta_q$ and on the type of the newborn at 
the coalescence time. 
Using Lemma \ref{lempcoal}, 
Equations \eqref{E''U} and \eqref{majcov}, and
 $r_K\leq 1$, we get
\begin{multline}
 \hat{\P}(TC^K(\beta_p,\beta_q)\neq -\infty, (A\beta_p)_{TC^K(\beta_p,\beta_q)}) \\\leq
\underset{k=1}{\overset{\lfloor \varepsilon K \rfloor-1}{\sum}}  
\Big[\sup_{n_A \in I_\eps^K} p_{AA}^{c_K}(n_A,k) + \sup_{n_A \in I_\eps^K} p_{Aa}^{c_K}(n_A,k)\Big]\hat{\E}[H_{k}^{K}]
+\sup_{n_A \in I_\eps^K} p_{aA}^{c_K}(n_A,k)\hat{\E}[U_{k}^{K}]
\leq  \frac{c}{K}\underset{k=1}{\overset{\lfloor \varepsilon K \rfloor-1}{\sum}}  \frac{1}{k},
\end{multline}
for a finite $c$, which is of order $\log K/K$.\\

\noindent \textit{Second step:} Then, we focus on the case where 
$\beta_p$ and $\beta_q$ are associated with an allele $a$ at the coalescence time.
The inclusion
\begin{eqnarray*}
 \Big\{\tilde{N}_a\Big(\tau_{TC^K(\beta_p,\beta_q)-1}^K\Big)=k, (a\beta_p, a\beta_q)_{TC^K(\beta_p,\beta_q)}\Big\}\subset  \underset{m=1}{\overset{J^K}{\bigcup}} 
\Big\{ TC^K(\beta_p,\beta_q)=m, \tilde{N}_a(\tau^K_{m-1})=k, (a\beta_p,a\beta_q)_m \Big\} ,
\end{eqnarray*}
and Equations (\ref{expr_reco1}) and \eqref{majcov} yield for every $k \in \{1,...,\lfloor \eps K\rfloor-1\}$:
\begin{equation*}\label{lem42} \hat{\P}\Big(\tilde{N}_a\Big(\tau_{TC^K(\beta_p,\beta_q)-1}^K\Big)=k, (a\beta_p, a\beta_q)_{TC^K(\beta_p,\beta_q)}\Big) \leq 
 \underset{n_A \in I_\eps^K}{\sup}\hspace{.2cm} p_{aa}^{c_K}(n_A,k)\hat{\E}[U_{k}^{K}]\leq \frac{4}{{s^2_-(\eps)}k(k+1)}. \end{equation*}
If $\beta_p$ and $\beta_q$ coalesce then undergo their first m-recombination when we look backward in time, and if the $a$-population has the size $k$
 at the coalescence time, it implies that the m-recombination occurs before the $\zeta^K_k$-th jump when we look forward in time.
 For $k,l < \lfloor \eps K\rfloor$,
\begin{multline*}
\hat{\P}\Big(\tilde{N}_a\Big(\tau^K_{TR_1^K(\beta_p)}\Big)=l,0 \leq TR_1^K(\beta_p) \leq TC^K(\beta_p,\beta_q)\Big| \tilde{N}_a\Big(\tau_{TC^K(\beta_p,\beta_q)-1}^K\Big)=k, (a\beta_p, a\beta_q)_{TC^K(\beta_p,\beta_q)}\Big)\\
 \leq  \sup_{n_A\in I_\eps^K} p^{r_K}_a(n_A,l)\Big( \mathbf{1}_{k>l}\hat{\E}  [U_{l}^{K}]+\mathbf{1}_{k\leq l}\hat{\E}  [U_{k,l}^{(K,1)}] \Big)
 \leq  
\frac{2r_K}{(l+1)s^2_-(\eps)}\Big( \mathbf{1}_{k>l}+ \frac{2\mathbf{1}_{k\leq l}\lambda_\eps^{l-k}}{s_-^5(\eps)(1-s_+(\eps))}
 \Big),
\end{multline*}
where the last inequality is a consequence of \eqref{rqpr1}, \eqref{E''U} and \eqref{majcov}. The two last equations finally yield the existence of a finite $c$ such that for every $K \in \N$:
\begin{equation*}   \hat{\P}(0 \leq TR_1^K(\beta_p) \leq TC^K(\beta_p,\beta_q), (a\beta_p, a\beta_q)_{TC^K(\beta_p,\beta_q)}) 
\leq  cr_K \underset{k,l=1}{\overset{\lfloor \varepsilon K \rfloor}{\sum}} \frac{ \mathbf{1}_{k>l}+ \mathbf{1}_{k\leq l}\lambda_\eps^{l-k}}{k(k+1)(l+1)} \leq cr_K,
\end{equation*}
which completes the proof of Lemma \ref{deuxreco} with Assumption \ref{condweak}. \end{proof}

\subsection{Probability to be descended from the first mutant} \label{bedescended}

We want to estimate the probability for the neutral lineage of $\beta_p$ to undergo no m-recombination. Recall Definition \eqref{defrhoK}:

\begin{lem}\label{lemvareta}
There exist two positive finite constants $c$ and $\eps_0$ such that for $\eps\leq \eps_0$:
$$ \limsup_{K \to \infty}\Big| \hat{\P}(TR^K_1(\beta_p)=-\infty)-(1-\rho_K) \Big|\leq c\eps^{1/2} .$$ 
\end{lem}

\begin{proof}
We introduce $\rho_m^K$, the conditional probability that the neutral lineage
 of $\beta_p$ m-recombines at time $\tau_m^K$, given $(\tilde{N}_A(\tau^K_n),\tilde{N}_a(\tau^K_n),n\leq J^K)$ and given that  
it has not m-recombined during the time interval $]\tau_m^K,T_\eps^K]$. 
The last condition implies that $\beta_p$ is associated with an allele $a$ at time $\tau_m^K$.
\begin{equation}
 \label{defrhoKm} \rho^K_m:=\mathbf{1}_{\{ \tilde{N}_a^K(\tau_m^K)-\tilde{N}_a^K(\tau_{m-1}^K)=1 \}} p_a^{r_K}(\tilde{N}_A^K(\tau_{m-1}^K),\tilde{N}_a^K(\tau_{m-1}^K)).
\end{equation}
 We also introduce $\eta^K$, the sum of these conditional 
probabilities for $1\leq m \leq J^K$:
\begin{equation*}
 \eta^K:={\sum_{m=1}^{J^K}}\rho^K_m.
\end{equation*}
We want to give a rigourous meaning to the sequence of equivalencies:
$$ \hat{\P}\Big(TR^K_1(\beta_p)=-\infty\Big|(\tilde{N}_A(\tau^K_m),\tilde{N}_a(\tau^K_m))_{m\leq J^K}\Big)=\prod_{m=1}^{J^K}(1-\rho^K_m) \sim \prod_{m=1}^{J^K}e^{-\rho^K_m}\sim e^{-\E[\eta^K]}, $$
when $K$ goes to infinity. Jensen's Inequality, the triangle inequality, and the Mean Value Theorem imply
\begin{multline}\label{ineg_tri} \hat{\E}\Big| \hat{\P}\Big(TR^K_1(\beta_p)=-\infty\Big|(\tilde{N}_A(\tau^K_m),\tilde{N}_a(\tau^K_m))_{m\leq J^K}\Big)
-(1-\rho_K) \Big|
\leq  \\
 \hat{\E}\Big|\hat{\P}\Big(TR^K_1(\beta_p)=-\infty\Big|(\tilde{N}_A(\tau^K_m),\tilde{N}_a(\tau^K_m))_{m\leq J^K}\Big)-e^{-\eta^K} \Big| 
+ \Big|e^{-\hat{\E}\eta^K}-(1-\rho_K) \Big|+ \hat{\E}\Big|\eta^K-\hat{\E}\eta^K \Big|  .\end{multline}
We aim to bound the right hand side of (\ref{ineg_tri}).
The bounding of the first term follows the method developed in Lemma 3.6 in \cite{schweinsberg2005random}. We refer to this proof, and get 
the following Poisson approximation
\begin{eqnarray}\label{bound1} \hat{\E}\Big|\hat{\P}\Big(TR^K_1(\beta_p)=-\infty\Big|(\tilde{N}_A(\tau^K_m),\tilde{N}_a(\tau^K_m))_{m\leq J^K}\Big)-e^{-\eta^K} \Big| &
 \leq &
 \sum_{k=1}^{\lfloor \varepsilon K \rfloor-1}\sup_{n_A \in I_\eps^K} \Big( p_a^{r_K}(n_A,k)\Big)^2\hat{\E}[U_{k}^{K}] \nonumber \\
&\leq &\frac{\pi^2 r_K^2}{3s_-^2(\eps)} ,\end{eqnarray}
where $I_\eps^K$ has been defined in \eqref{compact1} and the last inequality follows from \eqref{rqpr1} and \eqref{majcov}. 
To bound the second term, we need to estimate $\hat{\E}[\eta^K]$. Inequality \eqref{rqpr1} implies
\begin{equation}\label{infpra2}
 (1-c\eps)r_K \sum_{k=1}^{\lfloor \eps K \rfloor -1}\frac{U_k^K}{k+1}\leq \eta^K \leq r_K \sum_{k=1}^{\lfloor \eps K \rfloor -1}\frac{U_k^K}{k+1}.
\end{equation}
Adding \eqref{espnoreco} we get that for $\eps$ small enough,
\begin{equation}\label{secondterm}
\limsup_{K \to \infty} \Big|\exp(-\hat{\E}[\eta^K])-(1-\rho_K)\Big|\leq c\eps.
\end{equation}
The bounding of the last term of (\ref{ineg_tri}) requires a fine study of dependences between upcrossing numbers before and after the last visit 
to a given 
integer by the mutant population size. In
 particular, we widely use Equation (\ref{majcov}). We observe that $\hat{\E}|\eta^K-\hat{\E}\eta^K | \leq (\hat{\var}\hspace{.1cm}\eta^K)^{1/2}$, 
but the variance of $\eta^K$ is quite involved to study and according to Assumption \ref{condweak} and Equations \eqref{infpra2} and \eqref{majcov}, 
\begin{eqnarray}\label{tildeeta1}\Big|\hat{\var}\hspace{.1cm}\eta^K-\hat{\var}\hspace{.1cm}\Big( r_K \sum_{k=1}^{\lfloor \eps K \rfloor -1}\frac{U_k^K}{k+1} \Big)\Big|
  &\leq & c \varepsilon \hat{\E}\Big[\Big( r_K \sum_{k=1}^{\lfloor \eps K \rfloor -1}\frac{U_k^K}{k+1} \Big)^2\Big] \nonumber \\
& \leq &  c\eps r_K^2 \sum_{k,l=1}^{\lfloor \eps K \rfloor-1}\frac{\hat{\E}[(U_k^K)^2]+\hat{\E}[(U_l^K)^2]}{(k+1)(l+1)}
\leq c\eps,\end{eqnarray}
for a finite $c$ and $K$ large enough. Let $k \leq l < \lfloor \eps K\rfloor $,
and recall that by definition, $U_{l}^{K}=U_{k,l}^{(K,1)}+U_{k,l}^{(K,2)}$. Then we have
\begin{eqnarray*}
\Big|\hat{\cov}(U_{k}^{K},U_{l}^{K})\Big| \leq  \Big(\hat{\E}[(U_{k}^{K})^2]\hat{\E}[(U_{k,l}^{(K,1)})^2]\Big)^{1/2}+
\Big|\hat{\cov}(U_{k}^{K},U_{k,l}^{(K,2)})\Big|.
\end{eqnarray*}
Applying Inequalities \eqref{E''U} and \eqref{majcov} and noticing that $(1-s_-(\eps))<\lambda_\eps^{1/2}<1$ 
(recall the definition of $\lambda_\eps$ in \eqref{deflambda}) lead to 
\begin{eqnarray}\label{majcov2}
\Big|\hat{\cov}(U_{k}^{K},U_{l}^{K})\Big| \leq c( \lambda_\eps^{(l-k)/2}+\eps +(1-s_-(\eps))^{l-k}  )\leq 
c( \lambda_\eps^{(l-k)/2}+\eps)
\end{eqnarray}
for a finite $c$ and $\eps$ small enough. We finally get:
\begin{eqnarray}\label{varetatilde}
\hat{\var}\Big( r_K \sum_{k=1}^{\lfloor \eps K \rfloor -1}\frac{U_k^K}{k+1} \Big) & \leq & 2r_K^2  \sum_{k=1}^{\lfloor \varepsilon K \rfloor -1}
\frac{1}{k+1} \sum_{l=k}^{\lfloor \varepsilon K \rfloor -1} \frac{\hat{\cov}(U_{k}^{K},U_{l}^{K})}{l+1} \nonumber \\ 
& \leq &  cr_K^2\sum_{k=1}^{\lfloor \varepsilon K \rfloor -1}\frac{1}{k+1} \sum_{l=k}^{\lfloor \varepsilon K \rfloor -1} \frac{\lambda_\eps^{(l-k)/2}+\eps}{l+1} \leq  cr_K^2 {\eps}\log^2 K,
\end{eqnarray}
where we used \eqref{majcov2} for the second inequality.
Applying Jensen's Inequality to the left hand side of (\ref{ineg_tri}) and adding Equations 
(\ref{bound1}), (\ref{secondterm}), \eqref{tildeeta1} and (\ref{varetatilde}) we obtain
\begin{multline}\label{finL4.1} \hat{\E}\Big| \hat{\P}\Big(TR^K_1(\beta_p)=-\infty\Big|(\tilde{N}_A(\tau^K_m),\tilde{N}_a(\tau^K_m))_{m\leq J^K}\Big)
-(1-\rho_K) \Big|\\
\leq 
\frac{\pi^2 r_K^2}{3s_-^2(\eps)}+ c\eps+c(\eps+r_K^2\eps \log^2 K)^{1/2}. \end{multline}
This completes the proof of Lemma \ref{lemvareta}.
\end{proof}

We finally focus on the dependence between the genealogies of $\beta_p$ and $\beta_q$, 
and to this aim follow \cite{schweinsberg2005random} pp. 1622 to 1624 in the case $J=1$.
We define for $m \leq  J^K$ the random variable
$$ K_m=\mathbf{1}_{\{TR^K_1(\beta_p)\geq m\}}+\mathbf{1}_{\{TR^K_1(\beta_q)\geq m\}},$$
which counts the number of neutral lineages which recombine after the $m$-th jump (forward in time) among the lineages of $\beta_p$ and $\beta_q$.
First we will show that for $d \in\{0,1,2\}$,
\begin{multline}\label{maj1K0} \Big| \hat{\P}(K_0=d)-{d \choose 2}\ \hat{\E} \Big[\hat{\P}(TR^K_1(\beta_p)\geq 0|(\tilde{N}_A(\tau^K_m),\tilde{N}_a(\tau^K_m))_{m\leq J^K})^d \\
(1-\hat{\P}(TR^K_1(\beta_p)\geq 0|(\tilde{N}_A(\tau^K_m),\tilde{N}_a(\tau^K_m))_{m\leq J^K})^{2-d}\Big] \Big|\leq \frac{c}{\log K}, \end{multline}
for $\eps$ small enough and $K$ large enough, where $c$ is a finite constant.
The proof of this inequality can be found in \cite{schweinsberg2005random} pp. 1622-1624 and relies on Equation (\ref{lemme51}). The idea is to couple the process 
$(K_m,0\leq m\leq J^K)$ with a process $(K'_m,0\leq m\leq J^K)$ satisfying for every $m \leq J^K$,
$$ \mathcal{L}\Big( K'_{m-1}-K'_m | (\tilde{N}_A(\tau^K_m),\tilde{N}_a(\tau^K_m))_{m\leq J^K}, (K'_u )_{m\leq u \leq J^K} \Big)=Bin(2-K'_m, \rho^K_m), $$
where $Bin(n,p)$ denotes the binomial distribution with parameters $n$ and $p$, and $\rho_m^K$ has been defined in \eqref{defrhoKm}. This implies
$$ \mathcal{L}\Big( K'_{0} \Big| (\tilde{N}_A(\tau^K_m),\tilde{N}_a(\tau^K_m))_{m\leq J^K}  \Big)=Bin(2, \hat{\P}(TR^K_1(\beta_p)\geq 0|(\tilde{N}_A(\tau^K_m),\tilde{N}_a(\tau^K_m))_{m\leq J^K})), $$
and the coupling yields
$$ \hat{\P}(K'_m\neq K_m \text{ for some } 0\leq m\leq J^K)\leq c/\log K ,$$
for $\eps$ small enough and $K$ large enough, where $c$ is a finite constant. In particular, the weak dependence between two neutral lineages stated in Lemma \ref{deuxreco} is needed in this proof.
We now aim at proving that
\begin{multline} \Big| \hat{\E} [\hat{\P}(TR^K_1(\beta_p)\geq 0|(\tilde{N}_A(\tau^K_m),\tilde{N}_a(\tau^K_m))_{m\leq J^K})^d\\
(1-\hat{\P}(TR^K_1(\beta_p)\geq 0|(\tilde{N}_A(\tau^K_m),\tilde{N}_a(\tau^K_m))_{m\leq J^K})^{2-d}]- 
\rho_K^d(1-\rho_K)^{2-d} \Big|\leq c \eps^{1/2},\end{multline}
where we recall the definition of $\rho_K$ in \eqref{defrhoK}. Equation (\ref{lemme343344}) involves
\begin{multline*} \Big| \hat{\E} [\hat{\P}(TR^K_1(\beta_p)\geq 0|(\tilde{N}_A(\tau^K_m),\tilde{N}_a(\tau^K_m))_{m\leq J^K})^d(1-\hat{\P}(TR^K_1(\beta_p)\geq 0|(\tilde{N}_A(\tau^K_m),\tilde{N}_a(\tau^K_m))_{m\leq J^K}))^{2-d}]\\
- \rho_K^d(1-\rho_K)^{2-d} \Big|
 \leq 2 \hat{\E}|\hat{\P}(TR^K_1(\beta_p)\geq 0|(\tilde{N}_A(\tau^K_m),\tilde{N}_a(\tau^K_m))_{m\leq J^K})-\rho_K|.\end{multline*}
Applying Equation \eqref{finL4.1} and adding \eqref{maj1K0}, we finally get for $d$ in $\{0,1,2\}$,
 \begin{equation} \label{dpdce}
 \limsup_{K \to \infty} \Big| \hat{\P}(\mathbf{1}_{TR^K_1(\beta_p)\geq 0}+\mathbf{1}_{TR^K_1(\beta_p)\geq 0}=d) 
-{d \choose 2}  \rho_K^d(1-\rho_K)^{2-d} \Big|\leq c \eps^{1/2} .
 \end{equation}

\subsection{Neutral proportion at time $T_\eps^K$}

Let us again focus on the population process $N$. By abuse of notation, we still use $(TR_i^K(\beta_p), i \in \{1,2\})$ and 
$TC^K(\beta_p,\beta_q)$ to denote recombination and coalescence times of the neutral genealogies for the process $N$.
According to Lemma \ref{deuxreco}, Equation \eqref{dpdce}, and Coupling \eqref{couplage2},
  \begin{equation*}
 \limsup_{K \to \infty} {\P}\Big(\{TR_2^K(\beta_p)\geq 0\}\cup \{0 \leq TR^K_1(\beta_p) 
\leq TC^K(\beta_p,\beta_q)\}\Big|T_\eps^K<\infty\Big)\leq c\eps, \end{equation*}
and
 \begin{equation*}
 \limsup_{K \to \infty} \Big| {\P}(\mathbf{1}_{TR^K_1(\beta_p)\geq 0}+\mathbf{1}_{TR^K_1(\beta_p)\geq 0}=d|T_\eps^K<\infty) 
-{d \choose 2}  \rho_K^d(1-\rho_K)^{2-d} \Big|\leq c \eps^{1/2} ,
 \end{equation*}
for a finite $c$ and $\eps$ small enough.
Hence, it is enough to distinguish two cases for the randomly chosen neutral allele $\beta_p$: 
either its lineage has 
undergone one m-recombination, or no m-recombination.
 In the second case, $\beta_p$ is a $b_1$. In the first one, the probability that $\beta_p$ is a $b_1$ depends on the neutral 
proportion in the $A$ population at the coalescence time. We now state that this proportion stays nearly constant during the first period.

\begin{lem}\label{lemmajpro}
There exist two positive finite constants $c$ and $\eps_0$ such that for $\eps\leq \eps_0$,
\begin{equation*}
 \underset{K \to \infty}{\limsup}\hspace{.1cm} \P \Big( \underset{t \leq T^K_\eps}{\sup} \Big|P_{A,b_1}^K(t)-\frac{z_{Ab_1}}{z_A}\Big|>\sqrt{\varepsilon}, T^K_\eps<\infty \Big)  \leq c \varepsilon . 
\end{equation*}
\end{lem}

Lemma \ref{lemmajpro}, whose proof is postponed to Appendix \ref{prooflemma}, allows us to state the following lemma.

\begin{lem}\label{firstphaseweak}
There exist two positive finite constants $c$ and $\eps_0$ such that for $\eps\leq \eps_0$,
  \begin{equation*}
\limsup_{K \to \infty} \hat{\P}\Big( \Big| P^K_{a,b_2}(T^K_\eps) -  \frac{z_{Ab_2}}{z_A}\rho_K  \Big|
>\eps^{1/6} \Big)\leq c\eps^{1/6} .\end{equation*}
\end{lem}

\begin{proof}
The sequence $(\beta_i, i \leq \lfloor \eps K\rfloor)$ denotes the neutral alleles carried by the $a$-individuals at time $T_\eps^K$ and 
$$ A_2^K(i):= \{\beta_i \text{ has undergone exactly one m-recombination and is an allele } b_2\} .$$
If $\beta_i$ is a $b_2$, either its genealogy has undergone one m-recombination with an individual $Ab_2$, or 
it has undergone more than two m-recombinations. Thus 
\begin{equation*}
0\leq {N}^K_{ab_2}(T^K_\eps)- \sum_{i=1}^{\lfloor \eps K\rfloor} \mathbf{1}_{A_2^K(i)}\leq \sum_{i=1}^{\lfloor \eps K\rfloor} \mathbf{1}_{\{ TR^K_2(\beta_i)\neq -\infty \}}.
\end{equation*}
 Moreover, the probability of $A_2^K(i)$ depends on the neutral proportion in the $A$-population when $\beta_i$ m-recombines. For $i \leq \lfloor \eps K\rfloor$, 
\begin{equation}  \label{expl}\Big| \hat{\P} \Big(A^K_2(i)\Big|TR^K_1(\beta_i)\geq 0,TR^K_2(\beta_i)=-\infty,\sup_{t \leq T^K_\eps}
\Big|{P}^K_{A,b_1}(t)-\frac{z_{Ab_1}}{z_A}\Big|\leq\sqrt{\varepsilon} \Big) -\Big(1-\frac{z_{Ab_1}}{z_A}\Big)\Big| \leq \sqrt{\varepsilon}. \end{equation}
Lemma \ref{lemmajpro} and Equation \eqref{tildeTbiggerT} ensure that $\limsup_{K\to\infty} \hat{\P}({\sup}_{t \leq {T}^K_\eps} 
|{P}^K_{A,b_1}(t)-z_{Ab_1}/z_A|>\sqrt{\varepsilon})\leq c\eps$, and Lemmas \ref{deuxreco} and \ref{lemvareta}, and Coupling \eqref{couplage2} that
 $|\hat{\P}(TR^K_1(\beta_i)\geq 0,TR^K_2(\beta_i)=-\infty) -\rho_K|\leq c\eps$. It yields:
\begin{equation*} \Big| \hat{\P} \Big(TR^K_1(\beta_i)\geq 0,TR^K_2(\beta_i)=-\infty,\sup_{t \leq {T}^K_\eps}
\Big|{P}^K_{A,b_1}(t)-\frac{z_{Ab_1}}{z_A}\Big|\leq\sqrt{\varepsilon} \Big) -\rho_K\Big| \leq c\sqrt{\varepsilon} \end{equation*}
for a finite $c$ and $\eps$ small enough. Adding \eqref{expl} we get:
\begin{equation} \label{espprop} \limsup_{K \to \infty}\Big|\hat{\E}[{P}^K_{a,b_2}({T}^K_\eps)]-  \rho_K
\Big(1-\frac{z_{Ab_1}}{z_A}\Big) \Big|\leq c\sqrt{\eps}. \end{equation}
In the same way, using the weak dependence between lineages stated in \eqref{dpdce} and Coupling \eqref{couplage2}, we prove that 
$ \limsup_{K \to \infty}|\hat{\E}[{P}^K_{a,b_2}(T^K_\eps)^2]- \rho_K^2(1-z_{Ab_1}/z_A)^2 |\leq c\sqrt{\eps}.$ 
This implies, adding (\ref{espprop})
 that $ \limsup_{K \to \infty}\hat{\var}( {P}^K_{a,b_2}(T^K_\eps) ) \leq c \sqrt{\varepsilon}$. We end the proof by using Chebyshev's Inequality.
 \end{proof}

\subsection{Second and third periods}
Thanks to Lemma \ref{third_step} we already know that with high probability the neutral proportion in the $a$-population stays nearly constant 
during the third phase. We will prove that this is also the case during the second phase
This is due to the short duration of 
this period, which does not go to 
infinity with the carrying capacity $K$. 
\begin{lem}\label{lemma77}
There exist two positive finite constants $c$ and $\eps_0$ such that for $\eps\leq \eps_0$,
\begin{equation}
 \limsup_{K \to \infty} \hat{\P}\Big( \Big| P^K_{a,b_1}(T_{\textnormal{ext}}^K)-P^K_{a,b_1}(T_\eps^K) \Big|>\eps^{1/3}  \Big)\leq c\eps^{1/3}.
\end{equation}
\end{lem}

\begin{proof}
Let us introduce the stopping time $V_\eps^K$:
$$ V_\eps^K:=\inf \Big\{t \leq t_\eps, {\sup}_{\alpha \in\mathcal{A}}\Big|{N}_\alpha^K(T^K_\eps+t)/K-{n_\alpha}^{({N}(T^K_\eps)/K)}(t)\Big|>\eps^3 \Big\}, $$
where $t_\eps$ has been introduced in \eqref{tetaCfini1}.
Recall that $(\mathcal{F}_t^K, t \geq 0)$ denotes the canonical filtration of $N^K$. 
The Strong Markov Property, Doob's Maximal Inequality and Equation \eqref{crocheten1K1} yield:
\begin{multline*}
 \P\Big( {T_\eps^K \leq S^K_\eps},\sup_{t \leq t_\eps}|{M}_{a}^K(T^K_\eps+t\wedge V_\eps^K)-{M}_{a}^K(T^K_\eps)|>\sqrt{\eps} \Big)\\=  
\E\Big[ \mathbf{1}_{T_\eps^K \leq S^K_\eps}{\P}\Big(\sup_{t \leq t_\eps}|{M}_{a}^K(T^K_\eps+t\wedge V_\eps^K)-{M}_{a}^K(T^K_\eps)|>\sqrt{\eps}|\mathcal{F}_{T_\eps^K}\Big) \Big]\\ \leq 
\frac{1}{\eps}\E\Big[ \mathbf{1}_{T_\eps^K \leq S^K_\eps}\Big(\langle M_a \rangle_{T^K_\eps+t_\eps\wedge V_\eps^K}- \langle M_a \rangle_{T^K_\eps} \Big)\Big]
\leq \frac{t_\eps C(a,\bar{n}_a+\bar{n}_A)}{\eps (I(\Gamma,\eps)-\eps^3)K},
\end{multline*}
for $\eps$ small enough, where $I(\Gamma,\eps)$ has been defined in \eqref{tetaCfini1}. But according to Equation 
(\ref{result_champa1}) with $\delta = \eps^3$, 
$\limsup_{K \to \infty}\P(V_\eps^K< t_\eps| T_\eps^K \leq S^K_\eps)=0$.
 Moreover, Equations \eqref{defM} and \eqref{boundmart} imply for every $t \geq 0$
$$  \sup_{t \leq t_\eps}|{P}_{a,b_1}^K(T^K_\eps+t)-{P}_{a,b_1}^K(T^K_\eps)|\leq \sup_{t \leq t_\eps}|{M}_{a}^K(T^K_\eps+t)-{M}_{a}^K(T^K_\eps)|+r_Kt_\eps f_a .$$
As $r_K$ goes to $0$ under Assumption \ref{condweak}, we finally get:
\begin{equation} \label{lemmeteps} \underset{K \to \infty}{\limsup}\hspace{.1cm}\P \Big( \underset{t \leq t_\eps}{\sup}|{P}_{a,b_1}^K(T^K_\eps+t)-{P}_{a,b_1}^K(T^K_\eps)|>\sqrt{\varepsilon},T_\eps^K \leq
S^K_\eps \Big)=0. \end{equation}
Adding Lemma \ref{third_step} ends the proof of Lemma \ref{lemma77}.
\end{proof}

\subsection{End of the proof of Theorem \ref{main_result} in the weak recombination regime}
Thanks to Lemmas \ref{firstphaseweak} and \ref{lemma77} we get that for $\eps$ small enough,
$$ \limsup_{K \to \infty} \hat{\P}\Big( \Big| P^K_{a,b_2}(T^K_{\text{ext}}) -  \rho_K \frac{z_{Ab_2}}{z_A} \Big|>2\eps^{1/6} \Big)\leq c\eps^{1/6} .
 $$
Moreover, \eqref{gdeprobabfix} ensures that $\liminf_{K\to \infty}\P(T^K_{\eps} \leq  S^K_\eps|\text{Fix}^K)\geq 1-c\eps$, which implies
$$ \limsup_{K \to \infty} {\P}\Big(\mathbf{1}_{\text{Fix}^K} \Big| P^K_{a,b_2}(T^K_{\textnormal{ext}}) 
-  \rho_K \frac{z_{Ab_2}}{z_A} \Big|>2\eps^{1/6} \Big)\leq c\eps^{1/6} .
 $$
This is equivalent to the convergence in probability and ends the proof of Theorem \ref{main_result}.

 \renewcommand\thesection{\Alph{section}}
\setcounter{section}{0}
 \renewcommand{\theequation}{\Alph{section}.\arabic{equation}}

\section{Technical results}\label{known_results}
This section is dedicated to technical results needed in the proofs. 
We first present some results stated in \cite{champagnat2006microscopic}.
We recall Definitions \eqref{deffitinv1}, \eqref{defText}, \eqref{defTheta}, \eqref{SKeps}, \eqref{TKTKeps1} and \eqref{tetaCfini1} and that the notation $.^K$ refers to the processes that satisfy Assumption \ref{defrK}. Proposition \ref{fix_champ} is a direct 
consequence of 
Equations (42), (71), (72) and (74) in \cite{champagnat2006microscopic}:
\begin{pro}\label{fix_champ} 
There exist two posivite finite constants $M_1$ and $\eps_0$ such that for every $\eps\leq \eps_0$ 
\begin{equation}\label{taillepopfinale}
  \underset{K \to \infty}{\lim} \P\Big( \Big|N_a^K(T^K_{\textnormal{ext}})-K\bar{n}_a\Big|>\eps K \Big| \textnormal{Fix}^K \Big)=0,
\quad \text{and} \quad
 \underset{K \to \infty}{\limsup}\ \Big|\P(T^K_\eps<\infty)- \frac{S_{aA}}{f_a}\Big|\leq M_1\eps.
\end{equation}
Moreover there exists $M_2>0$ such that for every $\eps\leq \eps_0$, the probability of the event
\begin{equation}\label{def_FepsK} F_\eps^K=\Big\{T^K_\eps\leq S^K_\eps, N_A^K(T^K_\eps+t_\eps)<\frac{\eps^2 K}{2},
 | N_a^K(T^K_\eps+t_\eps)-\bar{n}_a K|<\frac{\eps K}{2}\Big\} \end{equation}
satisfies
\begin{equation}\label{res_champ}\underset{K \to \infty}{\liminf}\ \P(T^K_\eps\leq S^K_\eps)\geq \underset{K \to \infty}{\liminf}\ \P(F_\eps^K)\geq \frac{S_{aA}}{f_a}-M_2\eps, \end{equation}
and if $z \in \Theta$, then there exist two posivite finite constants $V$ and $c$ such that:
\begin{equation}\label{resSepsKz}
\underset{K \to \infty}{\liminf} \ \P({U}^K_\eps(z)>e^{VK})\geq 1-c\eps.
\end{equation}
\end{pro}

Thanks to these results we can state the following Lemma, which motivates the coupling of $N$ and $\tilde{N}$ and allows us to focus on the event $\{ \tilde{T}_\eps^K <\infty \}$ rather 
than on $\text{Fix}^K$ in Section \ref{proofweak}.

\begin{lem}
There exist two posivite finite constants $c$ and $\eps_0$ such that for every $\eps\leq \eps_0$ 
 \begin{equation} \label{tildeTbiggerT}
\underset{K \to \infty}{\limsup} \ \P(T_{\eps}^K<\infty,T^K_\eps>S^K_\eps)\leq c \eps,
\end{equation}
and 
\begin{equation}\label{gdeprobabfix}
 \limsup_{K\to \infty}\Big[\P(\{T^K_{\eps} \leq  S^K_\eps\} \setminus \textnormal{Fix}^K)+\P(\textnormal{Fix}^K \setminus \{T^K_{\eps} \leq  S^K_\eps\})\Big]\leq c\eps.
\end{equation}
\end{lem}

\begin{proof}
We have the following equality
\begin{eqnarray*}
 \P(T_{\eps}^K<\infty,T^K_\eps>S^K_\eps) &=& \P(T_{\eps}^K<\infty)-\P(T_{\eps}^K<\infty,T^K_\eps\leq S^K_\eps)\\
&=& \P(T_{\eps}^K<\infty)-\P(T^K_\eps\leq S^K_\eps),
\end{eqnarray*}
where we used the inclusion $\{T^K_\eps\leq S^K_\eps\}\subset \{T_{\eps}^K<\infty\}$, as $S_\eps^K$ is almost surely finite (a birth and death process with competition has a finite extinction time).
Equations \eqref{taillepopfinale} and \eqref{res_champ} ends the proof of \eqref{tildeTbiggerT}.
  From Equation \eqref{taillepopfinale}, we also have that for $\eps<\bar{n}_a/2$
\begin{equation}\label{Tepsfix}
\underset{K \to \infty}{\lim} \hspace{.1cm} \P(T_\eps^K=\infty|\textnormal{Fix}^K)\leq  \underset{K \to \infty}{\lim} 
\P\Big( \Big|N_a^K(T^K_{\text{ext}})-K\bar{n}_a\Big|>\eps K \Big| \textnormal{Fix}^K \Big)=0.
\end{equation}
This implies that 
$$ \P(T^K_\eps>S^K_\eps,\textnormal{Fix}^K)\leq \P(T^K_\eps>S^K_\eps,T^K_\eps<\infty)+\P(T^K_\eps=\infty,\textnormal{Fix}^K)\leq c\eps, $$
where we used \eqref{tildeTbiggerT}. Moreover, 
\begin{eqnarray*}
 \P(T^K_{\eps} \leq  S^K_\eps ,(\textnormal{Fix}^K)^c)&\leq & \P(T^K_{\eps} <\infty, (\textnormal{Fix}^K)^c)\\
& =& \P(T^K_{\eps} <\infty)- \P(T_\eps^K<\infty|\textnormal{Fix}^K)\P(\textnormal{Fix}^K)
 \leq  c\eps,
\end{eqnarray*}
where we used \eqref{taillepopfinale}, \eqref{Tepsfix} and \eqref{proba_fix}.
\end{proof}

We also recall some results on birth and death processes whose proofs can be found in Lemma 3.1 in
 \cite{schweinsberg2005random} and in \cite{athreya1972branching} p 109 and 112.

\begin{pro}
Let $Z=(Z_t)_{t \geq 0}$ be a birth and death process with individual birth and death rates $b$ and $d $. For $i \in \Z^+$, 
$T_i=\inf\{ t\geq 0, Z_t=i \}$ and $\P_i$ (resp. $\E_i$) is the law (resp. expectation) of $Z$ when $Z_0=i$. Then 
\begin{enumerate} 
 \item[$\bullet$] For $i \in \N$ and $t \geq 0$, 
\begin{equation} \label{espviemort}
 \E_i[Z_t]=ie^{(b-d)t}.
\end{equation}
 \item[$\bullet$] For $(i,j,k) \in \Z_+^3$ such that $j \in (i,k)$,
\begin{equation} \label{hitting_times1} \P_j(T_k<T_i)=\frac{1-(d/b)^{j-i}}{1-(d/b)^{k-i}} .\end{equation}
 \item[$\bullet$] If $d\neq b \in \R_+^*$, for every $i\in \Z_+$ and $t \geq 0$,
\begin{equation} \label{ext_times} \P_{i}(T_0\leq t )= \Big( \frac{d(1-e^{(d-b)t})}{b-de^{(d-b)t}} \Big)^{i}.\end{equation}
 \item[$\bullet$] If $0<d<b$, on the non-extinction event of $Z$, which has probability $1-(d/b)^{Z_0}$, the following convergence holds:
\begin{equation} \label{equi_hitting}  T_N/\log N \underset{N \to \infty}{\to} (b-d)^{-1}, \quad  a.s.  \end{equation}
\end{enumerate}
\end{pro}

Finally, we recall Lemma 3.4.3 in \cite{durrett2008probability} and Lemma 5.1 in \cite{schweinsberg2005random}. Let $d\in \N$. Then

\begin{lem}\label{lemmes_durrett}
\begin{enumerate}
 \item[$\bullet$] Let $a_1,...a_d$ and $b_1,...,b_d$ be complex numbers of modulus smaller than $1$. Then 
\begin{equation}\label{lemme343344} \Big|\underset{i=1}{\overset{d}{\prod}}a_i-\underset{i=1}{\overset{d}{\prod}}b_i\Big| \leq \underset{i=1}{\overset{d}{\sum}}|a_i-b_i| . \end{equation}


 \item[$\bullet$]Let $V$ and $V'$ be $\{0, 1,... , d\}$-valued random variables such that $\E[V] = \E[V']$. Then, there exist random variables $\tilde{V}$ and $\tilde{V}'$
 on some probability space such that $V$ and $\tilde{V}$ have the same distribution, $V$ and $\tilde{V}'$ have the same distribution, and
\begin{equation}\label{lemme51}\P( \tilde{V} \neq \tilde{V}') \leq  d \max\{\P( \tilde{V} \geq 2),\P( \tilde{V}'\geq 2)\}.\end{equation}
\end{enumerate}
\end{lem}

For $0<s<1$, if $\tilde{Z}^{(s)}$ denotes a random walk with jumps $\pm 1$ where up jumps occur
with probability $1/(2-s)$ and down jumps with probability $(1-s)/(2-s)$, we denote by $\P_i^{(s)}$ the law of $\tilde{Z}^{(s)}$ when the initial state is $i \in \N$ 
and introduce for every $a \in \R_+$ the stopping time
\begin{equation} \label{deftauas}
 \tau_a:=\inf \{ n \in \Z_+, \tilde{Z}^{(s)}_n= \lfloor a \rfloor \}.
\end{equation}
We also introduce for $\eps$ small enough and $0\leq j,k< \lfloor \eps K \rfloor$, the quantities
\begin{equation} \label{defqkl} q_{j,k}^{(s_1,s_2)}:=\frac{\P_{k+1}^{(s_1)}(\tau_{  \eps K }<\tau_k)}{\P_{k+1}^{(s_2)}(\tau_{  \eps K }<\tau_j)}=\frac{s_1}{1-(1-s_1)^{\lfloor \eps K \rfloor -k}} \frac{1-(1-s_2)^{\lfloor \eps K \rfloor -j}}{1-(1-s_2)^{k+1-j}}  , \quad  0<s_1,s_2<1,
\end{equation}
whose expressions are direct consequences of \eqref{hitting_times1}.
Let us now state a technical result, which helps us to control upcrossing numbers of
 the process $\tilde{N}_a^K$ before reaching the size $\lfloor\eps K\rfloor$ (see Appendix \ref{prooflemma}).

\begin{lem}\label{ineqqs}
For $a \in ]0,1/2[ $, $(s_1,s_2) \in [a,1-a]^2$, and $0\leq j \leq k < l < \lfloor \eps K \rfloor $,
\begin{equation} \label{minqk} 
 q^{(s_1 \wedge s_2, s_1 \vee s_2)}_{0,k} \geq s_1 \wedge s_2\quad \text{and} \quad 
 \Big| \frac{1}{q^{(s_1, s_2)}_{k,l}}-\frac{1}{q^{(s_2, s_1)}_{j,l}} \Big|\leq  \frac{4(1+{1}/{s_2})}{ea^2 |\log (1-a)|} |s_2-s_1|+ \frac{(1-s_2)^{l+1-k}}{s_2^3}.
\end{equation}
\end{lem}

\begin{proof}The first part of (\ref{minqk}) is a direct consequence of Definiton \eqref{defqkl}. 
Let $a$ be in $]0, 1/2[$ and consider functions $f_{\alpha,\beta} : x \mapsto (1-x^\alpha)/(1-x^{\beta}), (\alpha,\beta) \in \N^2, x \in [a,1-a] $.
 Then for $x \in [a,1-a]$,
\begin{equation} \label{lemtech} \| f_{\alpha,\beta} ' \|_\infty \leq 4(ea^2|\log (1-a)|)^{-1}. \end{equation}
Indeed, the first derivative of $f_{\alpha,\beta}$ is:
$$ f_{\alpha,\beta}'(x)=\frac{\beta x^{\beta-1}(1-x^{\alpha})-\alpha x^{\alpha-1}(1-x^{\beta})}{(1-x^{\beta})^2}. $$
Hence, for $x \in [a,1-a]$,
$$ |f_{\alpha,\beta}'(x)|\leq \frac{\beta (1-a)^{\beta}+\alpha (1-a)^{\alpha}}{(1-a)a^2}\leq 2\frac{\beta (1-a)^{\beta}+\alpha (1-a)^{\alpha}}{a^2} ,$$
where we used that $1-x^{\beta}\geq 1-(1-a)$ and that $1-a\geq 1/2$. Adding the following inequality 
$$\underset{k \in \N}{\sup}\hspace{.2cm}k(1-a)^k\leq\underset{x\in \R^+}{\sup}\hspace{.2cm}x(1-a)^x= (e|\log (1-a)|)^{-1},$$
completes the proof of \eqref{lemtech}.
From \eqref{hitting_times1}, we get for $0<s<1$ and $0\leq j\leq k<\lfloor \eps K \rfloor$,
\begin{eqnarray}\label{diff_k0} \Big|  \P_{l+1}^{(s)}(\tau_{ \eps K }<\tau_k)-\P_{l+1}^{(s)}(\tau_{ \eps K }<\tau_j)
 \Big|&=&\frac{(1-(1-s)^{k-j})((1-s)^{l+1-k}-(1-s)^{\lfloor \eps K \rfloor-k})}{(1-(1-s)^{\lfloor \eps K \rfloor-k})(1-(1-s)^{\lfloor \eps K \rfloor-j})} \nonumber \\
&\leq & (1-s)^{l+1-k}s^{-2}.
\end{eqnarray}
The triangle inequality leads to:
\begin{eqnarray*}
  \Big| \frac{1}{q^{(s_1, s_2)}_{k,l}}-\frac{1}{q^{(s_2, s_1)}_{j,l}} \Big|&=&\Big| \frac{\P_{l+1}^{(s_2)}(\tau_{  \eps K }<\tau_k)}
{\P_{l+1}^{(s_1)}(\tau_{  \eps K }<\tau_l)}-\frac{\P_{l+1}^{(s_1)}(\tau_{  \eps K }<\tau_j)}{\P_{l+1}^{(s_2)}(\tau_{  \eps K }<\tau_l)} \Big|\\
&\leq&\Big| \frac{1}{\P_{l+1}^{(s_1)}(\tau_{  \eps K }<\tau_l)}-\frac{1}{\P_{l+1}^{(s_2)}(\tau_{  \eps K }<\tau_l)}
\Big|\P_{l+1}^{(s_2)}(\tau_{  \eps K }<\tau_k)\\
&& +\frac{1}{\P_{l+1}^{(s_2)}(\tau_{  \eps K }<\tau_l)} \Big|\P_{l+1}^{(s_2)}(\tau_{  \eps K }<\tau_k)-\P_{l+1}^{(s_2)}(\tau_{  \eps K }<\tau_j)\Big|\\
&& +\frac{1}{\P_{l+1}^{(s_2)}(\tau_{  \eps K }<\tau_l)} \Big|\P_{l+1}^{(s_1)}(\tau_{  \eps K }<\tau_j)-\P_{l+1}^{(s_2)}(\tau_{  \eps K }<\tau_j)\Big|.
\end{eqnarray*}
Noticing that $ \P_{l+1}^{(s_2)}(\tau_{  \eps K }<\tau_l)\geq \P_{l+1}^{(s_2)}(\tau_{  \infty }<\tau_l) = \P_{1}^{(s_2)}(\tau_{  \infty }<\tau_0)=s_2 ,$ 
and using (\ref{diff_k0}) and the Mean Value Theorem with (\ref{lemtech}), we get the second part of (\ref{minqk}).
\end{proof}

\section{Proofs of Lemmas \ref{uphold} and \ref{lemmajpro}} \label{prooflemma}

\begin{proof}[Proof of Equation \eqref{majcov}]
In the whole proof, the integer $n_A$ denotes the state of $\tilde{N}_A$ and thus belongs to $I_\eps^K$ 
which has been defined in \eqref{compact1}. ${\P}_{(n_A,n_a)}$ (resp. $\hat{\P}_{(n_A,n_a)}$) denotes the probability ${\P}$ 
(resp. $\hat{\P}$) when $(\tilde{N}_A(0),\tilde{N}_a(0))=(n_A,n_a) \in \Z_+^2$. We introduce for $u \in \R_+$ the hitting time of $\lfloor u \rfloor$ by the process 
$\tilde{N}_a$:
\begin{equation}\label{tpssigma}
 \sigma^K_u:=\inf \{ t\geq 0, \tilde{N}_a^K(t)= \lfloor u \rfloor\}.
\end{equation}

Let $(i,j,k)$ be in $\Z_+^3$ with $j<k< \lfloor \eps K \rfloor$. Between jumps  $\zeta_j^K$ and $J^K$ the process $\tilde{N}_a$ necessarily 
jumps from 
$k$ to $k+1$. Then, either it reaches $\lfloor \eps K \rfloor$ before returning to $k$, either it again jumps from 
$k$ to $k+1$ and so on. Thus we approximate the probability that there is only one jump from $k$ to $k+1$ by comparing $U_{j,k}^{(K,2)}$ 
with geometrically distributed random variables. As we do not know the value of $\tilde{N}_A$ when $\tilde{N}_a$ hits $k+1$ for the first time, 
we take the maximum over all the possible values in $I_\eps^K$. Recall Definition \eqref{defhatP}. We get, 
as $\{\tilde{T}_{\eps}^K<\sigma_{j}^K\}\subset\{\tilde{T}_{\eps}^K<\sigma_{k}^K\}\subset\{\tilde{T}_{\eps}^K<\infty\}$:
\begin{eqnarray*}
 \hat{\P}(U_{j,k}^{(K,2)}=1|U_{j}^{K}=i)&\leq & \underset{n_A\in I_\eps^K}{\sup}\hat{\P}_{(n_A,k+1)}(\tilde{T}_{\eps}^K<\sigma_{k}^K|
\tilde{T}_{\eps}^K<\sigma_{j}^K,U_{j}^{K}=i)\\
&=&\underset{n_A\in I_\eps^K}{\sup}{\P}_{(n_A,k+1)}(\tilde{T}_{\eps}^K<\sigma_{k}^K|
\tilde{T}_{\eps}^K<\sigma_{j}^K,U_{j}^{K}=i)\\
&=&\underset{n_A\in I_\eps^K}{\sup}\frac{{\P}_{(n_A,k+1)}(\tilde{T}_{\eps}^K<\sigma_{k}^K,U_{j}^{K}=i)}
{{\P}_{(n_A,k+1)}(
\tilde{T}_{\eps}^K<\sigma_{j}^K,U_{j}^{K}=i)}\\
&=&\underset{n_A\in I_\eps^K}{\sup}\frac{{\P}_{(n_A,k+1)}(\tilde{T}_{\eps}^K<\sigma_{k}^K)}
{{\P}_{(n_A,k+1)}(
\tilde{T}_{\eps}^K<\sigma_{j}^K)},
\end{eqnarray*}
where we used that on the events $\{\tilde{T}_{\eps}^K<\sigma_{j}^K\}$ and $\{\tilde{T}_{\eps}^K<\sigma_{k}^K\}$ the jumps from $j$ to $j+1$ 
belong to the past, and Markov Property.
Coupling \eqref{couplage12} allows us to compare these conditional probabilities with the probabilities of the same 
events under $\P^{(s_-(\eps))}$ and $\P^{(s_+(\eps))}$, and recalling \eqref{defqkl} we get
$$ \hat{\P}(U_{j,k}^{(K,2)}=1|U_{j}^{K}=i)\leq  
\frac{\P_{k+1}^{(s_+(\eps))}(\tau_{\eps K}<\tau_{k})}
{\P_{k+1}^{(s_-(\eps))}(\tau_{\eps K}<\tau_j )}= q^{(s_+(\eps),s_-(\eps))}_{j,k}. $$
In an analogous way we show that $ \hat{\P}(U_{j,k}^{(K,2)}=1|U_{j}^{K}=i)\geq  q^{(s_-(\eps),s_+(\eps))}_{j,k}$. We deduce that we can construct 
two geometrically distributed random variables $G_1$ and $G_2$, possibly on an enlarged space, 
with respective parameters $q^{(s_+(\eps),s_-(\eps))}_{j,k}\wedge 1$ and $q^{(s_-(\eps),s_+(\eps))}_{j,k}$ such that on the event $\{ U_{j}^{K}=i \}$,
\begin{equation}\label{compaU'|Uj1} G_1\leq U_{j,k}^{(K,2)} \leq G_2.  \end{equation}
For the same reasons we obtain $ q^{(s_-(\eps),s_+(\eps))}_{j,k}\leq \hat{\P}(U_{j,k}^{(K,2)}=1)\leq q^{(s_+(\eps),s_-(\eps))}_{j,k} \wedge 1$, 
and again we can construct two random variables $G'_1\overset{d}{=}G_1$ and $G'_2\overset{d}{=}G_2$ such that
\begin{equation}\label{compaU'} G'_1 \leq U_{j,k}^{(K,2)} \leq G'_2. \end{equation}
Recall that $U_{0,k}^{(K,2)}=U_{k}^{K}$. Hence taking $j=0$ and adding the first part of Equation \eqref{minqk} give the first inequality of \eqref{majcov}.
According to Definition \eqref{def_s_-s_+1}, $|s_+(\eps)-s_-(\eps)|\leq c\eps$ for a finite $c$. 
Hence Equations \eqref{compaU'|Uj1}, \eqref{compaU'}  and \eqref{minqk} entail the existence of a finite  $c$ such that for $\eps$ small 
enough $| \hat{\E}[U_{j,k}^{(K,2)}|U_{j}^{K}=i ]-\hat{\E}[U_{j,k}^{(K,2)} ]| \leq  c\eps+ {(1-s_-(\eps))^{k+1-j}}/{s_-^3(\eps)}$. Thus according 
to the first part of Equation \eqref{majcov},
\begin{eqnarray} \label{cov1} \Big|\hat{\cov}(U_{j,k}^{(K,2)},U_{j}^{K})\Big| & \leq & \underset{i \in \N^*}{\sum} 
i\hat{\P}(U_{j}^{K}=i) \Big| \hat{\E}[U_{j,k}^{(K,2)}|U_{j}^{K}=i ]-\hat{\E}[U_{j,k}^{(K,2)} ] \Big| \nonumber \\
&\leq & \frac{2}{s^2_-(\eps)}\Big(c\eps+ \frac{(1-s_-(\eps))^{k+1-j}}{s_-^3(\eps)} \Big) ,\end{eqnarray}
where we use that $U_{j}^{K}\leq (U_{j}^{K})^2$. This ends the proof of \eqref{majcov}.
\end{proof}

\begin{proof}[Proof of Equation \eqref{E''U}]
Definitions \eqref{defdaba} and Coupling \eqref{couplage12} ensure that for $n_A \in I_\eps^K$, $\eps$ small enough and 
$K$ large enough, 
\begin{eqnarray*} \hat{\P}_{(n_A,k)}(\tilde{N}_a(dt)=k+1)
&=& \frac{\P_{(n_A,k+1)}(\tilde{T}_\eps^K<\infty )}
{\P_{(n_A,k)}(\tilde{T}_\eps^K<\infty )}{\P}_{(n_A,k)}(\tilde{N}_a(dt)=k+1)
\\&\geq &\frac{\P_{k+1}^{(s_-(\eps))}(\tau_{\eps K}<\tau_0)}{\P_{k}^{(s_+(\eps))}(\tau_{\eps K}<\tau_0)}f_a k dt\\
& = & \frac{1-(1-s_-(\eps))^{k+1}}{1-(1-s_-(\eps))^{\lfloor \eps K \rfloor}}
\frac{1-(1-s_+(\eps))^{\lfloor \eps K \rfloor}}{1-(1-s_+(\eps))^{k}}f_akdt\\
& \geq & s_-^2(\eps)f_akdt,
\end{eqnarray*}
and
\begin{eqnarray*} \hat{\P}_{(n_A,k)}(\tilde{N}_A(dt)\neq n_A)&\leq & 
\frac{\P_{k}^{(s_+(\eps))}(\tau_{\eps K}<\tau_0)}{\P_{k}^{(s_-(\eps))}(\tau_{\eps K}<\tau_0)}
{\P}_{(n_A,k)}(\tilde{N}_A(dt)\neq n_A)
\\ & \leq & (1+c\eps)2f_A\bar{n}_AKdt
\end{eqnarray*}
for a finite $c$, where we use \eqref{hitting_times1} and that $D_A+C_{A,A}\bar{n}_A=f_A$. Thus for $\eps$ small enough:
\begin{equation*}\label{holds}\hat{\P}(\tilde{N}_a(\tau_{m+1}^K)\neq \tilde{N}_a(\tau_{m}^K)| \tilde{N}_a(\tau_{m}^K)=k)\geq \frac{s_-^2(\eps)f_ak}{3f_A\bar{n}_AK}. \end{equation*}
If  $D_{k}^{K}$ denotes the downcrossing number from $k$ to $k-1$ before $\tilde{T}_\eps^K$, 
then under the probability $\hat{\P}$, we can bound $U_{k}^{K}+D_{k}^{K}+H_{k}^{K}$ by the sum of 
$U_{k}^{K}+D_{k}^{K}$ independent geometrically distributed random variables $G_i^K$ with parameter $ {s_-^2(\eps)f_ak}/{3f_A\bar{n}_AK}$ and 
$ H_{k}^{K}\leq {\sum}_{1\leq i\leq U_{k}^{K}+D_{k}^{K}}(G_i^K-1). $
Let us notice that if $k\geq 2$, $D_{k}^{K}=U_{k-1}^{K}-1$, and $D_{1}^{K}=0$. Using the first part of \eqref{majcov} twice we get 
$$ \hat{\E}[H_{k}^{K}] \leq \Big( \frac{4}{s^2_-(\eps)}-1 \Big) \Big(\frac{3f_A\bar{n}_AK}{s_-^2(\eps)f_ak}-1\Big),$$
which ends the proof of the first inequality in \eqref{E''U}.\\

\noindent As the mutant population 
size is not Markovian we cannot use symmetry and the Strong Markov Property to control the dependence of jumps before and after the last visit to a 
given state as in \cite{schweinsberg2005random}. Hence we describe the successive excursions of $\tilde{N}_a^K$ above a given level to 
get the last 
inequality in \eqref{E''U}. Let
$\tilde{U}_{j,k}^{(i)}$ be the number of jumps from 
$k$ to $k+1$ during the $i$th excursion above $j$. 
We first bound the expectation $\hat{\E}[ (\tilde{U}_{j,k}^{(i)})^2 ]$. During an excursion above $j$, $\tilde{N}_a$ hits $j+1$, but 
we do not know the value of $\tilde{N}_A$ at this time. Thus we take the maximum value for the probability when $n_A$ belongs to 
$I_\eps^K$, and 
$$\hat{\P} (\tilde{U}_{j,k}^{(i)}\geq 1 ) \leq {\sup}_{n_A\in I_\eps^K} \hat{\P}_{(j+1,n_A)} 
( \sigma_{k+1}^K<\sigma_j^K|\sigma_j^K<\tilde{T}_{\eps}^K).$$
 Then using Coupling \eqref{couplage12} and Definition \eqref{defhatP} we obtain
\begin{eqnarray*}
{\hat\P} \Big(\tilde{U}_{j,k}^{(i)}\geq 1 \Big) & \leq & 
\sup_{n_A\in I_\eps^K} \frac{\P_{(j+1,n_A)}(\sigma_{k+1}^K<\sigma_j^K<\tilde{T}_{\eps}^K<\infty)}
{\P_{(j+1,n_A)}(\sigma_j^K<\tilde{T}_{\eps}^K<\infty)} 
\\
& \leq &  \frac{\P_{j}^{(s_+(\eps))}(\tau_{ \eps K }<\tau_{0})\P_{k+1}^{(s_-(\eps))}(\tau_{j}<\tau_{ \eps K })\P_{j+1}^{(s_+(\eps))}(\tau_{k+1}<\tau_{ j })}
{\P_{j}^{(s_-(\eps))}(\tau_{ \eps K }<\tau_{0})\P_{j+1}^{(s_+(\eps))}(\tau_{j}<\tau_{ \eps K })}.
\end{eqnarray*}
Adding Equation \eqref{hitting_times1} we finally get 
\begin{equation}\label{uijk1} \hat{\P} \Big(\tilde{U}_{j,k}^{(i)}\geq 1 \Big) \leq \frac{(1-s_-(\eps))^{k+1-j}}{s_-(\eps)(1-s_+(\eps))} . \end{equation}
Moreover if $\tilde{U}_{j,k}^{(i)}\geq 1$, $\tilde{N}_a$ necessarily hits $k$ after its first jump from $k$ to $k+1$, and before its return to $j$. 
Using the same techniques as before we get:
\begin{eqnarray*}
 \hat{\P} \Big(\tilde{U}_{j,k}^{(i)}= 1 |\tilde{U}_{j,k}^{(i)}\geq 1 \Big)&\geq & \underset{n_A \in I_\eps^K}{\inf} 
 \hat{\P}_{(n_A,k)}(\sigma_{j}^K<\sigma_{k+1}^K |\sigma_{j}^K<\tilde{T}_{\eps}^K ) \\
&\geq &
 \frac{\P_{j}^{(s_-(\eps))}(\tau_{ \eps K }<\tau_{0})\P_{k}^{(s_+(\eps))}(\tau_{j}<\tau_{k+1})}
{\P_{j}^{(s_+(\eps))}(\tau_{ \eps K }<\tau_{0})\P_{k}^{(s_-(\eps))}(\tau_{j}<\tau_{\eps K })},
\end{eqnarray*}
which yields
\begin{equation}\label{uijk2} \hat{\P} \Big(\tilde{U}_{j,k}^{(i)}= 1 |\tilde{U}_{j,k}^{(i)}\geq 1 \Big)
 \geq s_-(\eps)s_+(\eps)\Big(\frac{1-s_+(\eps)}{1-s_-(\eps)}\Big)^{k-j}\geq s_-^2(\eps)\Big(\frac{1-s_+(\eps)}{1-s_-(\eps)}\Big)^{k-j}=:q.\end{equation}
Hence, given that $\tilde{U}_{j,k}^{(i)}$ is non-null, $\tilde{U}_{j,k}^{(i)}$ is smaller than a geometrically distributed random variable with 
parameter $q$. In particular, 
$$ \E \Big[ (\tilde{U}_{j,k}^{(i)})^2 |\tilde{U}_{j,k}^{(i)}\geq 1 \Big]\leq \frac{2}{q^2}= \frac{2}{s^4_-(\eps)}\Big(\frac{1-s_-(\eps)}{1-s_+(\eps)}\Big)^{2(k-j)}. $$
Adding Equation (\ref{uijk1}) and recalling that $|s_+(\eps)-s_-(\eps)|\leq c\eps$ for a finite $c$ yield
$$ \hat{\E} \Big[(\tilde{U}_{j,k}^{(i)})^2\Big]\leq \frac{2 \lambda_\eps^{k-j}}{ s^5_-(\eps)(1-s_+(\eps))} , \quad \text{where} \quad \lambda_\eps:=\frac{(1-s_-(\eps))^3}{(1-s_+(\eps))^2} <1.$$
Using that for $n \in \N$ and $(x_i,1\leq i \leq n) \in \R^n$, $(\sum_{1\leq i \leq n}x_i)^2 \leq n\sum_{1\leq i \leq n}x_i^2$ and that the number of excursions above $j$ before $\tilde{T}_\eps^K$ is $U_{j}^{K}-1$, we get
$$ \hat{\E} \Big[({U}_{j,k}^{(K,1)})^2\Big]\leq \hat{\E} \Big[U_{j}^{K}-1\Big] \frac{2 \lambda_\eps^{k-j}}{ s^5_-(\eps)(1-s_+(\eps))} \leq \frac{4 \lambda_\eps^{k-j}}{ s^7_-(\eps)(1-s_+(\eps))}, $$
where we used the first part of Equation \eqref{majcov}. This ends the proof of Equation \eqref{E''U}.\end{proof}

\begin{proof}[Proof of Equation \eqref{espnoreco}] 
Definition \eqref{defqkl}, Inequality \eqref{compaU'} and Equation \eqref{hitting_times1} yield:
\begin{equation*} \label{exprexpli}  r_K \sum_{k=1}^{\lfloor \eps K \rfloor -1} \frac{\hat{\E}[U_k^K]}{k+1} \geq r_K\underset{k=1}{\overset{\lfloor \eps K \rfloor -1}{\sum}}
\Big[(k+1)q_{0,k}^{(s_+(\eps),s_-(\eps))}\Big]^{-1}=\frac{r_K (A-B)}{s_+(\eps)({1-(1-s_-(\eps))^{\lfloor \eps K \rfloor}})}
, \end{equation*}
with
$$ A:=\underset{k=1}{\overset{\lfloor \eps K \rfloor -1}{\sum}}\frac{1-(1-s_-(\eps))^{k+1}}{k+1}, \quad \text{and} \quad
  B:=(1-s_+(\eps))^{\lfloor \eps K \rfloor}\underset{k=1}{\overset{\lfloor \eps K \rfloor -1}{\sum}}\frac{1-(1-s_-(\eps))^{k+1}}{(1-s_+(\eps))^{k}(k+1)}. $$
For large $K$, $A=\log (\eps K)+O(1)$, and for every $u>1$ there exists $D(u)<\infty$ such that 
$\sum_{k=1}^{\lfloor \eps K \rfloor}u^k/(k+1)\leq D(u)u^{\lfloor \eps K \rfloor }/\lfloor \eps K \rfloor $. This implies that
$B \leq {c}/{\lfloor \varepsilon K \rfloor}$ for a finite $c$.
Finally, by definition, for $\eps$ small enough, 
$|s_+(\eps)-S_{aA}/f_a|\leq c\eps$ for a finite constant $c$. This yields 
$$  r_K \sum_{k=1}^{\lfloor \eps K \rfloor -1} \frac{\hat{\E}[U_k^K]}{k+1}\geq (1-c\eps) \frac{r_Kf_a \log K }{S_{aA}}$$
for a finite $c$ and concludes the proof for the lower bound. The upper bound is obtained in the same way. This ends the proof of Lemma \ref{uphold}.
\end{proof}

\begin{proof}[Proof of Lemma \ref{lemmajpro}]
We use Coupling \eqref{couplage11} to control the growing of the mutant population during the first period 
of invasion, and the semi-martingale decomposition in Proposition \ref{mart_prop} to bound the fluctuations of $M_A$. 
The hitting time of $\lfloor \eps K \rfloor$ and non-extinction event of $Z^*_\eps$ are denoted by:
\begin{equation*}\label{hittingextetoile}
 T^{*,K}_\eps=\inf \{ t \geq 0, Z^*_\eps(t)= \lfloor \varepsilon K \rfloor \}, \quad \text{and} \quad F^*_\eps=\Big\{   Z^*_\eps(t)\geq 1, \forall t\geq 0  \Big\}, \quad *\in \{-,+\}.\end{equation*}
{Let us introduce the difference of probabilities 
$$B_\eps^K:=\P\Big(\sup_{t \leq T^K_\eps}
\Big|P^K_{A,b_1}(t)-\frac{z_{Ab_1}}{z_A}\Big|>\sqrt{\varepsilon}, T^K_\eps <\infty\Big)-\P\Big(\sup_{t \leq T^K_\eps}
\Big|P^K_{A,b_1}(t)-\frac{z_{Ab_1}}{z_A}\Big|>\sqrt{\varepsilon}, F^-_\eps,  T^K_\eps \leq S^K_\eps  \Big).$$
Then $B_\eps^K$ is nonnegative and we have
\begin{eqnarray*}
 \label{majA} B_\eps^K &\leq &
\P(T^K_\eps <\infty  )-\P(F^-_\eps,  T^K_\eps \leq S^K_\eps  )\\
&=& \P(T^K_\eps <\infty  )-\P( T^K_\eps \leq S^K_\eps  )+\P(T^{(+,K)}_\eps<\infty,  T^K_\eps \leq S^K_\eps  )-\P(F^-_\eps,  T^K_\eps \leq S^K_\eps  ),
\end{eqnarray*}
where the inequality comes from the inclusion $\{F^-_\eps,  T^K_\eps \leq S^K_\eps\} \subset \{T^K_\eps <\infty \}$, as $S_\eps^K$ is almost surely finite.
The equality is a consequence of Coupling \eqref{couplage11} 
which ensures that on the event $\{ T^K_\eps \leq S^K_\eps \}$, $\{T^{(+,K)}_\eps<\infty\}$ holds.
By noticing that 
$$ \{F^-_\eps , T^K_\eps \leq S^K_\eps \} \subset  \{ T^{(-,K)}_\eps <\infty ,T^K_\eps \leq S^K_\eps \} \subset \{ T^{(+,K)}_\eps <\infty ,T^K_\eps \leq
 S^K_\eps \} $$
we get the bound 
\begin{equation}
 \label{majA} B_\eps^K \leq
 \P(T^K_\eps <\infty  )-\P( T^K_\eps \leq S^K_\eps  )+\P(T^{(+,K)}_\eps<\infty  )-\P(F^-_\eps).
\end{equation}
The values of the two first probabilities are approximated in \eqref{taillepopfinale} and \eqref{res_champ},
and \eqref{hitting_times1} implies that $\P(T^{+,K}_\eps<\infty)-\P(F^-_\eps)=s_+(\eps)/(1-(1-s_+(\eps))^{\lfloor \eps K \rfloor})-s_-(\eps)$. Hence 
\begin{equation}\label{restri} \underset{K \to \infty}{\limsup} \ B_\eps^K \leq c\eps, \end{equation} 
where $c$ is finite for $\eps$ small enough, which allows us to focus on the intersection with the event $\{F^-_\eps,T^K_\eps 
\leq S^K_\eps \}$.} We recall that $|N_{Ab_1}N_{ab_2}-N_{Ab_2}N_{ab_1}|\leq N_A N_a$, and that Assumption \ref{condweak} holds. Then \eqref{defM} 
and \eqref{TKTKeps1} imply for $\eps$ small enough 
$$\underset{t \leq T^K_\eps \wedge S_\eps^K }{\sup}\Big|P_{A,b_1}(t)-\frac{z_{Ab_1}}{z_A}-M_A(t)\Big|\leq r_Kf_a  T^K_\eps
\underset{t \leq T^K_\eps \wedge S_\eps^K }{\sup} \left\{ \frac{N_a(t)}{N_A(t)} \right\}\leq 
\frac{r_K f_a \varepsilon  T^{K}_\eps}{\bar{n}_A-2\varepsilon {C_{A,a}}/{C_{A,A}}} \leq \frac{c\eps T^{K}_\eps}{\log K},$$
for a finite $c$. Moreover, $F^-_\eps \cap \{ T^K_\eps \leq S^K_\eps \} \subset F^-_\eps \cap \{ T^K_\eps \leq T_\eps^{(-,K)} \}$. Thus we get 
$$\P\Big(\sup_{t \leq T^K_\eps}\Big|P_{A,b_1}(t)-\frac{z_{Ab_1}}{z_A}-M_A(t)\Big|>\frac{\sqrt{\varepsilon}}{2},F^-_\eps,T^K_\eps \leq S^K_\eps\Big)\leq 
\P \Big( \frac{c \varepsilon  T^{(-,K)}_\eps}{\log K}>\sqrt{\varepsilon}/2,F^-_\eps\Big).$$
Finally, Equation (\ref{equi_hitting}) ensures that $\lim_{K \to \infty}T^{(-,K)}_\eps/\log  K  = s_-(\eps)^{-1}$ a.s.  on the non-extinction event $F^-_\eps$.
 Thus for $\eps$ small enough,
\begin{equation} \label{partie1}\lim_{K \to \infty}\P\Big(\sup_{t \leq T^K_\eps}\Big|P_{A,b_1}(t)-\frac{z_{Ab_1}}{z_A}-M_A(t)\Big|>\frac{\sqrt{\varepsilon}}{2},F^-_\eps,T^K_\eps \leq 
S^K_\eps\Big) = 0 . \end{equation}
To control the term $|M_A|$, we introduce the sequence of real numbers $t_K=(2f_a\log K )/S_{aA}$:
\begin{eqnarray*} \label{decoup_pro} \P\Big(\underset{t \leq T^K_\eps}{\sup}|M_A(t)|>\frac{\sqrt{\varepsilon}}{2},F^-_\eps,T^K_\eps \leq 
S^K_\eps\Big)\leq \P\Big(\underset{t \leq T^K_\eps}{\sup}|M_A(t)|>\frac{\sqrt{\varepsilon}}{2},T^K_\eps \leq S^K_\eps \wedge t_K \Big)+\P(T^K_\eps > t_K,F^-_\eps).\end{eqnarray*}
Equation \eqref{def_s_-s_+1} yields for $\eps$ small enough, $t_K.s_-(\eps)/\log K>3/2$. Thus thanks to (\ref{equi_hitting}) we get,
\begin{equation*} \label{star}\lim_{K \to \infty}  \P(T^{K}_\eps > t_K,F^-_\eps)\leq \lim_{K \to \infty} \P(T^{-,K}_\eps > t_K,F^-_\eps)= 0 .\end{equation*}
Applying Doob's maximal inequality to the submartingale $|M_A|$ and \eqref{crocheten1K1} we get:
\begin{eqnarray*}\label{partie2}
 \P(\underset{t \leq T^K_\eps}{\sup}|M_A(t)|>\sqrt{\varepsilon}/2, T^K_\eps \leq S^K_\eps\wedge t_K ) & \leq  & \P(\underset{t \leq t_K}{\sup}|M_A(t \wedge T^K_\eps 
\wedge S^K_\eps)|>\sqrt{\varepsilon}/2 ) \nonumber\\
& \leq & \frac{4}{\varepsilon}\E\Big[ \langle M_A\rangle_{t_K \wedge T^K_\eps \wedge S^K_\eps} \Big]\\
&\leq  & \frac{4}{\eps}\frac{C(A,2\bar{n}_A) t_K}{(\bar{n}_A-2\eps C_{A,a}/C_{A,A})K},
\end{eqnarray*}
which goes to $0$ at infinity. Adding Equation \eqref{partie1} leads to:
$$ \lim_{K \to \infty}\P\Big(\sup_{t \leq T^K_\eps}\Big|P_{A,b_1}(t)-\frac{z_{Ab_1}}{z_A}\Big|>{\sqrt{\varepsilon}},F^-_\eps,T^K_\eps \leq 
S^K_\eps\Big)=0. $$
Finally, Equation \eqref{restri} complete the proof of Lemma~\ref{lemmajpro}.
\end{proof}

{\bf Acknowledgements:} {\sl The author would like to thank Jean-François Delmas and Sylvie M\'el\'eard for their help and their careful 
reading of this paper. She also wants to thank Sylvain Billiard and Pierre Collet for fruitful discussions during her work and 
several suggestions, so as the two anonymous referees for several corrections and improvements. This work  was partially funded by project MANEGE `Mod\`eles
Al\'eatoires en \'Ecologie, G\'en\'etique et \'Evolution'
09-BLAN-0215 of ANR (French national research agency) and Chair Mod\'elisation Math\'ematique et Biodiversit\'e VEOLIA-Ecole Polytechnique-MNHN-F.X.}

\bibliographystyle{abbrv}

\begin{thebibliography}{}

\end{thebibliography}


\begin{thebibliography}{10}

\bibitem{athreya1972branching}
K.~B. Athreya and P.~E. Ney.
\newblock {\em Branching processes}, volume~28.
\newblock Springer-Verlag Berlin, 1972.

\bibitem{barrett2008adaptation}
R.~D. Barrett and D.~Schluter.
\newblock Adaptation from standing genetic variation.
\newblock {\em Trends in Ecology \& Evolution}, 23(1):38--44, 2008.

\bibitem{barton1998effect}
N.~H. Barton.
\newblock The effect of hitch-hiking on neutral genealogies.
\newblock {\em Genetical Research}, 72(2):123--133, 1998.

\bibitem{billiard2013stochastic}
S.~Billiard, R.~Ferri{\`e}re, S.~M{\'e}l{\'e}ard, and V.~C. Tran.
\newblock Stochastic dynamics of adaptive trait and neutral marker driven by
  eco-evolutionary feedbacks.
\newblock {\em arXiv preprint arXiv:1310.6274}, 2013.

\bibitem{champagnat2006microscopic}
N.~Champagnat.
\newblock A microscopic interpretation for adaptive dynamics trait substitution
  sequence models.
\newblock {\em Stochastic Processes and their Applications}, 116(8):1127--1160,
  2006.

\bibitem{champagnat2014adaptation}
N.~Champagnat, P.-E. Jabin, and S.~M{\'e}l{\'e}ard.
\newblock Adaptation in a stochastic multi-resources chemostat model.
\newblock {\em Journal de Math{\'e}matiques Pures et Appliqu{\'e}es},
  101(6):755--788, 2014.

\bibitem{champagnat2011polymorphic}
N.~Champagnat and S.~M{\'e}l{\'e}ard.
\newblock Polymorphic evolution sequence and evolutionary branching.
\newblock {\em Probability Theory and Related Fields}, 151(1-2):45--94, 2011.

\bibitem{collet2011rigorous}
P.~Collet, S.~M{\'e}l{\'e}ard, and J.~A. Metz.
\newblock A rigorous model study of the adaptive dynamics of mendelian
  diploids.
\newblock {\em Journal of Mathematical Biology}, pages 1--39, 2011.

\bibitem{coop2012patterns}
G.~Coop and P.~Ralph.
\newblock Patterns of neutral diversity under general models of selective
  sweeps.
\newblock {\em Genetics}, 192(1):205--224, 2012.

\bibitem{coron2013slow}
C.~Coron.
\newblock Slow-fast stochastic diffusion dynamics and quasi-stationary
  distributions for diploid populations.
\newblock {\em arXiv preprint arXiv:1309.3405}, 2013.

\bibitem{coron2014stochastic}
C.~Coron.
\newblock Stochastic modeling of density-dependent diploid populations and the
  extinction vortex.
\newblock {\em Advances in Applied Probability}, 46(2):446--477, 2014.

\bibitem{durand2010standing}
E.~Durand, M.~I. Tenaillon, C.~Ridel, D.~Coubriche, P.~Jamin, S.~Jouanne,
  A.~Ressayre, A.~Charcosset, and C.~Dillmann.
\newblock Standing variation and new mutations both contribute to a fast
  response to selection for flowering time in maize inbreds.
\newblock {\em BMC evolutionary biology}, 10(1):2, 2010.

\bibitem{durrett2008probability}
R.~Durrett.
\newblock {\em Probability models for DNA sequence evolution}.
\newblock Springer, 2008.

\bibitem{durrett2004approximating}
R.~Durrett and J.~Schweinsberg.
\newblock Approximating selective sweeps.
\newblock {\em Theoretical population biology}, 66(2):129--138, 2004.

\bibitem{eriksson2008accurate}
A.~Eriksson, P.~Fernstr{\"o}m, B.~Mehlig, and S.~Sagitov.
\newblock An accurate model for genetic hitchhiking.
\newblock {\em Genetics}, 178(1):439--451, 2008.

\bibitem{etheridge2006approximate}
A.~Etheridge, P.~Pfaffelhuber, and A.~Wakolbinger.
\newblock An approximate sampling formula under genetic hitchhiking.
\newblock {\em The Annals of Applied Probability}, 16(2):685--729, 2006.

\bibitem{EK}
S.~Ethier and T.~Kurtz.
\newblock Markov processes: Characterization and convergence, 1986, 1986.

\bibitem{fournier2004microscopic}
N.~Fournier and S.~M{\'e}l{\'e}ard.
\newblock A microscopic probabilistic description of a locally regulated
  population and macroscopic approximations.
\newblock {\em The Annals of Applied Probability}, 14(4):1880--1919, 2004.

\bibitem{hermisson2005soft}
J.~Hermisson and P.~S. Pennings.
\newblock Soft sweeps molecular population genetics of adaptation from standing
  genetic variation.
\newblock {\em Genetics}, 169(4):2335--2352, 2005.

\bibitem{hermisson2008pattern}
J.~Hermisson and P.~Pfaffelhuber.
\newblock The pattern of genetic hitchhiking under recurrent mutation.
\newblock {\em Electron J Probab}, 13(68):2069--2106, 2008.

\bibitem{ikeda1989stochastic}
N.~Ikeda and S.~Watanabe.
\newblock Stochastic differential equations and diffusion processes.
\newblock 1989.

\bibitem{kaplan1989hitchhiking}
N.~L. Kaplan, R.~Hudson, and C.~Langley.
\newblock The" hitchhiking effect" revisited.
\newblock {\em Genetics}, 123(4):887--899, 1989.

\bibitem{stephanie2009selective}
S.~Leocard.
\newblock Selective sweep and the size of the hitchhiking set.
\newblock {\em Advances in Applied Probability}, 41(3):731--764, 2009.

\bibitem{smith1974hitch}
J.~Maynard~Smith and J.~Haigh.
\newblock The hitch-hiking effect of a favourable gene.
\newblock {\em Genet Res}, 23(1):23--35, 1974.

\bibitem{mcvean2007structure}
G.~McVean.
\newblock The structure of linkage disequilibrium around a selective sweep.
\newblock {\em Genetics}, 175(3):1395--1406, 2007.

\bibitem{messer2013population}
P.~W. Messer and D.~A. Petrov.
\newblock Population genomics of rapid adaptation by soft selective sweeps.
\newblock {\em Trends in ecology \& evolution}, 28(11):659--669, 2013.

\bibitem{metz1996adaptive}
J.~A. Metz, S.~A. Geritz, G.~Mesz{\'e}na, F.~J. Jacobs, and J.~Van~Heerwaarden.
\newblock Adaptive dynamics, a geometrical study of the consequences of nearly
  faithful reproduction.
\newblock {\em Stochastic and spatial structures of dynamical systems},
  45:183--231, 1996.

\bibitem{neuhauser1997genealogy}
C.~Neuhauser and S.~M. Krone.
\newblock The genealogy of samples in models with selection.
\newblock {\em Genetics}, 145(2):519--534, 1997.

\bibitem{ohta1975effect}
T.~Ohta and M.~Kimura.
\newblock The effect of selected linked locus on heterozygosity of neutral
  alleles (the hitch-hiking effect).
\newblock {\em Genetical research}, 25(03):313--325, 1975.

\bibitem{orr2001haldane}
H.~A. Orr and A.~J. Betancourt.
\newblock Haldane's sieve and adaptation from the standing genetic variation.
\newblock {\em Genetics}, 157(2):875--884, 2001.

\bibitem{pennings2006soft}
P.~S. Pennings and J.~Hermisson.
\newblock Soft sweeps ii-molecular population genetics of adaptation from
  recurrent mutation or migration.
\newblock {\em Molecular biology and evolution}, 23(5):1076--1084, 2006.

\bibitem{pennings22006soft}
P.~S. Pennings and J.~Hermisson.
\newblock Soft sweeps iii: the signature of positive selection from recurrent
  mutation.
\newblock {\em PLoS genetics}, 2(12):e186, 2006.

\bibitem{peter2012distinguishing}
B.~M. Peter, E.~Huerta-Sanchez, and R.~Nielsen.
\newblock Distinguishing between selective sweeps from standing variation and
  from a de novo mutation.
\newblock {\em PLoS genetics}, 8(10):e1003011, 2012.

\bibitem{pfaffelhuber2007approximating}
P.~Pfaffelhuber and A.~Studeny.
\newblock Approximating genealogies for partially linked neutral loci under a
  selective sweep.
\newblock {\em Journal of mathematical biology}, 55(3):299--330, 2007.

\bibitem{prezeworski2005signature}
M.~Prezeworski, G.~Coop, and J.~D. Wall.
\newblock The signature of positive selection on standing genetic variation.
\newblock {\em Evolution}, 59(11):2312--2323, 2005.

\bibitem{schweinsberg2005random}
J.~Schweinsberg and R.~Durrett.
\newblock Random partitions approximating the coalescence of lineages during a
  selective sweep.
\newblock {\em The Annals of Applied Probability}, 15(3):1591--1651, 2005.

\bibitem{stephan2006hitchhiking}
W.~Stephan, Y.~S. Song, and C.~H. Langley.
\newblock The hitchhiking effect on linkage disequilibrium between linked
  neutral loci.
\newblock {\em Genetics}, 172(4):2647--2663, 2006.

\bibitem{stephan1992effect}
W.~Stephan, T.~H. Wiehe, and M.~W. Lenz.
\newblock The effect of strongly selected substitutions on neutral
  polymorphism: analytical results based on diffusion theory.
\newblock {\em Theoretical Population Biology}, 41(2):237--254, 1992.

\bibitem{wallace1975hard}
B.~Wallace.
\newblock Hard and soft selection revisited.
\newblock {\em Evolution}, pages 465--473, 1975.

\end{thebibliography}

\end{document}